\documentclass[11pt,a4paper,leqno]{article}
\usepackage{a4wide}
\usepackage{latexsym}
\usepackage{amsmath}
\usepackage{amssymb}
\usepackage{stackrel} 
\usepackage{color}
\usepackage{graphicx}

\usepackage{hyperref}

\newtheorem{defin}{Definition}
\newtheorem{lemma}{Lemma}
\newtheorem{prop}{Proposition}
\newtheorem{theo}{Theorem}

\newenvironment{proof}{\medskip\par\noindent{\bf Proof}}{\hfill $\Box$
\medskip\par}

\newcommand{\C}{\mathbb{C}}
\newcommand{\N}{\mathbb{N}}
\newcommand{\R}{\mathbb{R}}

\begin{document}
\title{Boundary layer expansions for initial value problems with two complex time variables}
\author{{\bf A. Lastra\footnote{The author is partially supported by the project MTM2016-77642-C2-1-P of Ministerio de Econom\'ia y Competitividad, Spain}, S. Malek\footnote{The author is partially supported by the project MTM2016-77642-C2-1-P of Ministerio de Econom\'ia y Competitividad, Spain.}}\\
University of Alcal\'{a}, Departamento de F\'{i}sica y Matem\'{a}ticas,\\
Ap. de Correos 20, E-28871 Alcal\'{a} de Henares (Madrid), Spain,\\
University of Lille 1, Laboratoire Paul Painlev\'e,\\
59655 Villeneuve d'Ascq cedex, France,\\
{\tt alberto.lastra@uah.es}\\
{\tt Stephane.Malek@math.univ-lille1.fr }}
\date{}
\maketitle
\thispagestyle{empty}
{ \small \begin{center}
{\bf Abstract}
\end{center}

We study a family of partial differential equations in the complex domain, under the action of a complex perturbation parameter $\epsilon$. We construct inner and outer solutions of the problem and relate them to asymptotic representations via Gevrey asymptotic expansions with respect to $\epsilon$, in adequate domains. The construction of such analytic solutions is closely related to the procedure of summation with respect to an analytic germ, put forward in~\cite{mosch18}, whilst the asymptotic representation leans on the cohomological approach determined by Ramis-Sibuya Theorem.

\medskip

}
\bigskip \bigskip

\section{Introduction}

The main aim in this work is to describe the analytic solutions and asymptotic behavior of the solutions of a family of initial value problems in the complex domain. Such family consists of partial differential equations in two complex time variables of the form
\begin{equation}
Q(\partial_z)u(t_1,t_2,z,\epsilon)=P(t_1^{k_1+1}\partial_{t_1}, t_2^{k_2+1}\partial_{t_2}, \partial_z,z,\epsilon)u(t_1,t_2,z,\epsilon)+f(t_1,t_2,z,\epsilon),\label{epralintro}
\end{equation}
under given initial data $u(0,t_2,z,\epsilon)\equiv u(t_1,0,z,\epsilon)\equiv 0$. Here, $Q(X)\in\C[X]$ and $P(T_1,T_2,Z,z,\epsilon)$ stands for a polynomial in $(T_1,T_2,Z)$ with holomorphic coefficients w.r.t. $(z,\epsilon)$ on $H_{\beta}\times D(0,\epsilon_0)$, where $H_{\beta}$ stands for the horizontal strip in the complex plane 
$$H_\beta:=\{z\in \C:|\hbox{Im}(z)|<\beta\},$$ 
for some $\beta>0$, and $D(0,\epsilon_0)\subseteq\C$ stands for the open disc centered at the origin with radius $\epsilon_0$, for some small $\epsilon_0>0$. The symbol $\epsilon$ acts as a small complex perturbation parameter in the equation. Moreover, $k_1,k_2$ are positive ingers with $1\le k_1< k_2$. The forcing term, constructed in detail in Section~\ref{sec1}, turns out to be a holomorphic function in $\C^2\times H_{\beta'}\times D(0,\epsilon_0)$. In this paper, we also adopt the notation $\overline{D}(0,r)$ for the closed disc centered at $0\in\C$ and radius $r>0$.

The precise constrains involving the parameters involved in each of the equations determining the family of PDEs under study is described in Section~\ref{sec1}.

\vspace{0.3cm}

This is the continuation of a series of works devoted to the study of PDEs in the complex domain under the action of two complex time variables. In~\cite{family1}, the authors have studied a family of nonlinear initial value Cauchy problems of the form
\begin{multline}\label{e11}
Q(\partial_z)\partial_{t_1}\partial_{t_2}u(t_1,t_2,z,\epsilon)=(P_1(\partial_z,\epsilon)u(t_1,t_2,z,\epsilon))(P_2(\partial_z,\epsilon)u(t_1,t_2,z,\epsilon))\\
+P(t_1,t_2,\partial_{t_1},\partial_{t_2},\partial_z,\epsilon)u(t_1,t_2,z,\epsilon)+f(t_1,t_2,z,\epsilon),
\end{multline} 
where the terms in $Q,P,P_1,P_2$ are such that the action of $t_1$ and $t_2$ is symmetric. Moreover, we assume that the polynomial 
\begin{equation}\label{e92}
\mathcal{P}(t_1,t_2,\partial_{t_1},\partial_{t_2},\partial_z,\epsilon):=Q(\partial_z)\partial_{t_1}\partial_{t_2}-L_1(t_1,t_2,\partial_{t_1},\partial_{t_2},\partial_z,\epsilon)
\end{equation}
where $L_1$ involves leading terms of the differential operator $P$, can be factorized in such a way that each of the factors only depend on one of the times variables, i.e.
$$\mathcal{P}(t_1,t_2,\partial_{t_1},\partial_{t_2},\partial_z,\epsilon)=\mathcal{P}_1(t_1,\partial_{t_1},\partial_z,\epsilon)\mathcal{P}_2(t_2,\partial_{t_2},\partial_z,\epsilon).$$
From this symmetric configuration one is able to construct families of analytic bounded solutions $u_{d_h}(t_1,t_2,z,\epsilon)\in\mathcal{O}_b(\mathcal{T}_1\times\mathcal{T}_2\times H_{\beta''}\times\mathcal{E}_h)$ for every $0\le h\le \iota-1$, where $\mathcal{T}_1,\mathcal{T}_2,\mathcal{E}_j$ stand for open finite sectors with vertex at the origin in $\C$, $\beta''>0$, and $(\mathcal{E}_h)_{0\le h\le \iota-1}$ is a good covering of $\C^\star$ (see Definition~\ref{defi4}). Moreover, an asymptotic behavior of such solutions can be observed with respect to the perturbation parameter $\epsilon$. Indeed, there exists a formal power series $\epsilon\mapsto\hat{u}(t_1,t_2,z,\epsilon)\in\mathbb{E}[[\epsilon]]$, where $\mathbb{E}$ stands for the Banach space of holomorphic and bounded functions defined in $\mathcal{T}_1\times\mathcal{T}_2\times H_{\beta''}$ with the norm of the supremum, which turns out to be a formal solution of (\ref{e11}). In addition to this, a multisummability result joins both, analytic and formal solutions (see~\cite{family1}, Theorem 2).

\vspace{0.3cm}

In the second study~\cite{family2}, the property of symmetry of the equations drops, and $\mathcal{P}$ in (\ref{e92}) is no longer factorizable into two terms which only present dependence on one of the time variables. This asymmetry causes that the procedure followed in~\cite{family1} is no longer valid in that second framework and the procedure followed differs from that in~\cite{family1}. 

In both studies, a Borel-Laplace method is applied. In the symmetric case, the analytic solution is constructed as the Laplace transform with respect to $\tau_1$ and $\tau_2$ of an auxiliary function $\omega(\tau_1,\tau_2)$, which is well defined in a domain of the form $(S_1\cup D(0,\rho))\times (S_2\cup D(0,\rho))$, for some $\rho>0$ and certain sectors $S_1,S_2$ with vertex at the origin. Moreover, such function admits an exponential growth at infinity with respect to $\tau_1\in S_1$ and $\tau_2\in S_2$. This is the suitable configuration in order to apply Borel-Laplace techniques on each of the variables involved and achieve summability results (see Section~\ref{sec51}). On the other hand the asymmetric settings in the problem considered in~\cite{family2} causes the function $\omega(\tau_1,\tau_2)$ only be defined in sets of the form $S_1\times (S_2\cup D(0,\rho))$, and a small divisor phenomena is observed. Therefore, the summability conditions are not satisfied and a different approach has to be followed, focused on studying the natural domains and asymptotic behavior of $\omega(\tau_1,\tau_2)$, and apply summability results asymmetrically.

\vspace{0.3cm}

In this sense, this work is concerned with a family of equations in which none of the previous strategies is satisfactory. On the one hand, the symmetric situation does not hold in the present work, so the strategy followed in~\cite{family1} is not available. On the other hand, the strategy considered in~\cite{family2} does not apply because the auxiliary function $\omega(\tau_1,\tau_2)$ requires that, at least for one of the variables, a neighborhood of the origin is contained in its domain of definition. This is not the case, so a summability procedure can not be followed. The reason for failure is that the deformation path accomplished when computing the difference of two consecutive solutions of the main problem, written as the Laplace transform $\omega$, is no longer applicable. A small divisor phenomena occurs which does not allow to determine the Gevrey orders involved in the relationship between the analytic and the formal solution. The precise reasoning on the failure of this procedure is detailed in Section~\ref{secfirstapproach}.

\vspace{0.3cm}

A second novelty in the present work is the appearance of two different kinds of families of analytic solutions of the main problem for which one can give a picture of their asymptotic behavior with respect to the perturbation parameter. Following the terminology in the study of boundary layer solutions of equations, we distinguish the inner solutions (see Section~\ref{secin})  and the outer solutions (see Section~\ref{secout}) of the main problem, and describe their asymptotic representation with respect to the perturbation parameter near the origin. A recent work by the authors~\cite{lamaboundarylayer} constructs boundary layer expansions for certain initial value problem with merging turning points, regarding inner and outer solutions, which only considers the action of one time variable in the equation. In that previous work and also in the present work, the Gevrey orders of the asymptotic representation of the inner and outer solutions are different in general. As mentioned, we observe a comparable phenomenon in the present situation. However, in our context the inner solutions might not be $\lambda_1k_1-$summable in general, for some $\lambda_1>0$ to be precised.

The so-called inner and outer expansions are of great interest in mathematics, under the theory of matched asymptotic expansions. For a detailed theory on this subject we refer to classical textbooks such as~\cite{beor,ec,la,om,sk,wa}. For the general aspects on Gevrey asymptotic expansions in this context, we refer to the book~\cite{frsc}.

\vspace{0.3cm}

The study of singularly perturbed PDEs in the complex domain is a topic of increasing interest. In 2015, H. Yamazawa and M. Yoshino~\cite{yayo} studied parametric Borel summability in semilinear systems of PDEs of fuchsian type, 
 and of combined irregular and fuchsian type by M. Yoshino~\cite{yo}.

\vspace{0.3cm}

The theory of monomial summability was put forward by M. Canalis-Durand, J. Mozo-Fern\'andez and R. Sch\"afke in~\cite{camosch}. Recently, S.A. Carrillo and J. Mozo-Fern\'andez have studied further properties on monomial summability and Borel-Laplace methods on this theory in~\cite{camo1,camo2}, and this technique has been successfully applied to families of singularly perturbed ODEs and PDEs. A step further is given by J. Mozo-Fern\'andez and R. Sch\"afke in~\cite{mosch18}, where the authors put forward novel Gevrey asymptotic expansions and summability with respect to an analytic germ, and apply their technique to different families of PDEs and ODEs. In the present work, we make use of a similar approach to search for solutions of our main problem in the form of a Laplace-like transform with a meromorphic kernel at 0 (see Remark at page~\pageref{aaa} and (\ref{e285})).

\vspace{0.3cm}

We now give a general overview of the sections in which the present study is divided, and the main results obtained. The statement of the main problem under consideration is settled in Section~\ref{sec1}, where we give arguments on the reason of failure of the methods used in~\cite{family1,family2} in this family of PDEs (see Section~\ref{secfirstapproach}) and determine the shape of the analytic solution as a Laplace-like transform of a function related to the meromorphic kernel $\Omega$ provided in (\ref{e285}) (see Section~\ref{subsec2}): given two good coverings of $\C^\star$, $(\mathcal{E}_{h_1}^0)_{0\le h_1\le \iota_1-1}$ and $(\mathcal{E}_{h_2}^\infty)_{0\le h_2\le \iota_2-1}$, with the first good covering consisting of sectors with wide enough opening (see Definition~\ref{defi4}), we construct two sets of analytic solutions of (\ref{epralintro}) in the form
\begin{multline*}
u_{\xi_{h_j}}(t_1,t_2,z,\epsilon)\\
:=\frac{1}{(2\pi)^{1/2}}\int_{-\infty}^{\infty}\int_{L_{d_{h_j}}}\omega(u,m,\epsilon)\exp\left(-\left(\frac{u}{\epsilon^{\lambda_1}t_1}\right)^{k_1}-\left(\frac{u}{\epsilon^{\lambda_2}t_2}\right)^{k_2}\right)\exp(izm)\frac{du}{u}dm, \quad j=1,2.
\end{multline*}
The elements of the first family are constructed on a domain of the form
$\mathcal{T}_1\times(\mathcal{T}_2\cap D(0,\rho_2))\times H_{\beta'}\times\mathcal{E}^0_{h_1}$ for some $\rho_2>0$. The elements in the second family are constructed on domains of the form $\mathcal{T}_1\times\mathcal{T}_{2,\epsilon}\times H_{\beta'}\times\mathcal{E}^\infty_{h_2}$. Here, $\mathcal{T}_1$ is a finite sector, $\mathcal{T}_2$ is an infinite sector and the direction $\xi_h\in\R$ is an appropriate argument, the set $\mathcal{T}_{2,\epsilon}$ is a bounded sector which depends on $\epsilon\in\mathcal{E}^{\infty}_{h_2}$ and tends to infinity with $\epsilon\to 0$; $\lambda_1,\lambda_2\in\N$. The function $\omega(u,m,\epsilon)$ comes as a result of a fixed point argument in the Borel plane, in certain Banach spaces of functions (see Section~\ref{subsec3} and Section~\ref{subsec4}). We finally relate the analytic solutions to an asymptotic representation in different subdomains achieving the construction and asymptotic results on the inner solutions (see Section~\ref{secin}) and on the outer solutions (see Section~\ref{secout}). In both situations, we provide differences of consecutive solutions (in the sense that they are associated to consecutive sectors in the fixed good covering), and apply Ramis-Sibuya theorem (see Theorem (RS)) to arrive at the existence of a common asymptotic representation of all inner solutions and an asymptotic representation of all the outer solutions. These asymptotic behavior appears in the form of Gevrey asymptotic expansions of order $1/(\lambda_1 k_1)$ with respect to the perturbation parameter $\epsilon$ on $\mathcal{E}_{h_2}^\infty$ regarding the inner solutions, whilst $\lambda_2k_2-$Gevrey summability can be observed regarding the outer solutions, with respect to the perturbation parameter (see Theorem~\ref{lema495b}).

\vspace{0.3cm}

The paper is organized as follows.\\
In Section~\ref{seclap}, we recall some definitions and results on Laplace transform, asymptotic expansions and Fourier transform. In Section~\ref{sec1}, we state the main problem under study and analyze different ways to approach the problem. The section ends with the construction of the solution to an auxiliary problem in the Borel plane within a Banach space of functions with exponential growth and decay. Section~\ref{sec4} is devoted to the construction of the analytic solution of the main problem in addition to the inner and outer solutions. The work ends in Section~\ref{lastsec} with the study of the parametric Gevrey asymptotic expansions of both types of solutions in appropriate domains, with respect to the perturbation parameter. The last section is focused on the technical proof of Lemma~\ref{lema3}, left at the end of the paper for the sake of clarity in the ongoing argumentation.

\section{Laplace transform, asymptotic expansions and Fourier transform}\label{seclap}

In this section, we recall the main definitions and results involved in the theory of Borel summable formal power series with coefficients in a fixed complex Banach space $( \mathbb{E}, \left\|\cdot\right\|_{\mathbb{E}} )$. For a detailed description on the classical theory, we refer to \cite{ba}, Section 3.2. For the sake of simplicity, we have decided to make use of a slightly modified version of the classical theory. This slightly modified version of the $k-$Borel transform has been used in previous works by the authors such as~\cite{lama,lama1,lama2}, and in the study of singularly perturbed families of equations such as~\cite{family1}.

\begin{defin} Let $k \geq 1$ be an integer. For every $n\ge1$, we define $ m_{k}(n):= \Gamma(\frac{n}{k})$. A formal power series $\hat{f}(t) = \sum_{n=1}^{\infty}  f_{n}t^{n} \in t\mathbb{E}[[t]]$ is said to be $m_{k}-$summable with respect to $t$ in the direction $d \in [0,2\pi)$ if \medskip

{\bf i)} there exists $\rho>0$ such that the formal power series, 
$$ \mathcal{B}_{m_k}(\hat{f})(\tau) = \sum_{n=1}^{\infty} \frac{ f_{n} }{ \Gamma(\frac{n}{k}) } \tau^{n}
\in \tau\mathbb{E}[[\tau]],$$
known as the formal $m_{k}-$Borel transform of $\hat{f}$, is absolutely convergent in $D(0,\rho)$. \medskip

{\bf ii)} there exists $\delta > 0$ such that the function $\mathcal{B}_{m_k}(\hat{f})(\tau)$, which in principal is only defined on some neighborhood of the origin, can be analytically continued with
respect to $\tau$ in a sector
$S_{d,\delta} = \{ \tau \in \mathbb{C}^{\ast} : |d - \mathrm{arg}(\tau) | < \delta \} $. Moreover, there exist $C,K >0$
such that
$$ \left\|\mathcal{B}_{m_k}(\hat{f})(\tau)\right\|_{\mathbb{E}}
\leq C e^{ K|\tau|^{k} },\quad \tau \in S_{d, \delta}.$$
\end{defin}

Under the previous hypotheses, the vector valued Laplace transform of $\mathcal{B}_{m_k}(\hat{f})(\tau)$ in the direction $d$ is defined by
$$ \mathcal{L}^{d}_{m_k}(\mathcal{B}_{m_k}(\hat{f}))(t) = k \int_{L_{\gamma}}
\mathcal{B}_{m_k}(\hat{f})(u) e^{ - ( u/t )^{k} } \frac{d u}{u}.$$
The integration path consists of the half-line $L_{\gamma} = (0,\infty)e^{\gamma\sqrt{-1}} \subset S_{d,\delta} \cup \{ 0 \}$, where $\gamma$ depends on
$T$ and is chosen in such a way that $\cos(k(\gamma - \mathrm{arg}(t))) \geq \delta_{1} > 0$, for some fixed $\delta_{1}$.
The function $\mathcal{L}^{d}_{m_k}(\mathcal{B}_{m_k}(\hat{f}))(t)$ is well defined, holomorphic and bounded in any sector
$$ S_{d,\theta,R^{1/k}} = \{ T \in \mathbb{C}^{\ast} : |t| < R^{1/k} \ \ , \ \ |d - \mathrm{arg}(t) | < \theta/2 \},$$
where $\frac{\pi}{k} < \theta < \frac{\pi}{k} + 2\delta$ and
$0 < R < \delta_{1}/K$. This function is called the $m_{k}-$sum of the formal series $\hat{f}(t)$ in the direction $d$.\medskip

\noindent We use the notation $\mathcal{B}_{m_k,t}$ (resp. $\mathcal{L}^d_{m_k,t}$) to emphasize that $t$ is the variable with respect to which Borel (resp. Laplace) transform is applied, if necessary.

The following result holds regarding the properties of the $m_{k}-$Borel transform.

\begin{prop}\label{prop8}  Let $\hat{f}(t) = \sum_{ n \geq 1} f_{n}t^{n}\in\mathbb{E}[[t]]$, $\hat{g}(t) = \sum_{n \geq 1} g_{n}t^{n}\in\mathbb{E}[[t]]$. Let $k \geq 1$ be an integer number. Then, the following formal identity holds:
\begin{equation}
\mathcal{B}_{m_k}(t^{k+1}\partial_{t}\hat{f}(t))(\tau) = k \tau^{k} \mathcal{B}_{m_k}(\hat{f}(t))(\tau) \label{Borel_diff}
\end{equation}
\end{prop}

We also recall some classical properties of inverse Fourier transform, which are used in our construction.
\begin{defin} Let $\beta, \mu \in \mathbb{R}$. We denote by
$E_{(\beta,\mu)}$ the vector space of continuous functions $h : \mathbb{R} \rightarrow \mathbb{C}$ such that
$$ \left\|h(m)\right\|_{(\beta,\mu)} = \sup_{m \in \mathbb{R}} (1+|m|)^{\mu} \exp( \beta |m|) |h(m)| $$
is finite. The space $E_{(\beta,\mu)}$ equipped with the norm $\left\|\cdot\right\|_{(\beta,\mu)}$ is a Banach space.
\end{defin}

\begin{prop}\label{prop359}
Let $f \in E_{(\beta,\mu)}$ with $\beta > 0$, $\mu > 1$. The inverse Fourier transform of $f$, given by
$$ \mathcal{F}^{-1}(f)(x) = \frac{1}{ (2\pi)^{1/2} } \int_{-\infty}^{+\infty} f(m) \exp( ixm ) dm,\quad x\in\mathbb{R},$$
can be extended to an analytic function on the horizontal strip
\begin{equation}
H_{\beta} = \{ z \in \mathbb{C} / |\mathrm{Im}(z)| < \beta \}. \label{strip_H_beta}
\end{equation}
Let $\phi(m) := im f(m) \in E_{(\beta,\mu - 1)}$. Then, the following statements hold:
\begin{itemize}
\item[a)] $\partial_{z} \mathcal{F}^{-1}(f)(z) = \mathcal{F}^{-1}(\phi)(z)$, for $z \in H_{\beta}$.
\item[b)] Let $g\in E_{(\beta,\mu)}$, and consider the convolution product of $f$ and $g$, 
$$\psi(m):=\frac{1}{(2\pi)^{1/2}}\int_{-\infty}^{\infty}f(m-m_1)g(m_1)dm_1.$$
Then, $\psi\in E_{(\beta,\mu)}$ and $\mathcal{F}^{-1}(f)(z)\mathcal{F}^{-1}(g)(z)=\mathcal{F}^{-1}(\psi)(z),\quad z\in H_\beta$.
\end{itemize}
\end{prop}

\section{Statement of the main problem and solution of an auxiliary problem}\label{sec1}

Our main aim in this work is to provide analytic and formal solutions to the main problem under study (\ref{epralintro}), and give information about the asymptotic behavior relating both. In this section, we detail the elements involved in the main problem under study, and provide different approaches which might be followed in order to search analytic and asymptotically related formal solutions.

Let $1\le k_1<k_2$, and $D_1,D_2\ge 2$ be integer numbers. We also fix $\lambda_1,\lambda_2\in\N$. For $1\le \ell_1\le D_1$ and $1\le \ell_2\le D_2$, we consider non negative integers $\delta_{\ell_1},\delta_{\ell_2}$ and $\Delta_{\ell_1\ell_2}$. 

We assume that 
\begin{equation}\label{e193}
\Delta_{D_1D_2}=\lambda_1k_1\delta_{D_1}+\lambda_2k_2\delta_{D_2},\quad \lambda_2k_2>\lambda_1k_1,
\end{equation}
\begin{multline}
\Delta_{\ell_1\ell_2}> \lambda_1k_1\delta_{\ell_1}+ \lambda_2k_2\delta_{\ell_2},\quad k_1\delta_{D_1}+ k_2\delta_{D_2}\ge k_1\delta_{\ell_1}+ k_2\delta_{\ell_2},\\ 
1\le \ell_1\le D_1-1,1\le \ell_2\le D_2-1.\label{e209}
\end{multline}


We consider polynomials with complex coefficients $Q, R_{D_1D_2}$, and $R_{\ell_1\ell_2}$, for every $1\le \ell_1\le D_1-1$ and $1\le \ell_2\le D_2-1$. We assume that 
\begin{equation}\label{e203}
\frac{Q(im)}{R_{D_1D_2}(im)}\in A_{Q,R_{D_1D_2}},\quad m\in\R,
\end{equation}
where $A_{Q,R_{D_1},R_{D_2}}$ stands for the sectorial annulus
$$A_{Q,R_{D_1D_2}}=\left\{z\in \C:r^1_{Q,R_{D_1D_2}}\le|z|\le r^2_{Q,R_{D_1D_2}},\quad\hbox{arg}(z)\in (\alpha_{Q,R_{D_1D_2}},\beta_{Q,R_{D_1D_2}})\right\},$$
for some $0<r^1_{Q,R_{D_1D_2}}<r^2_{Q,R_{D_1D_2}}$ and $\alpha_{Q,R_{D_1D_2}},\beta_{Q,R_{D_1D_2}}\in\R$, with $\alpha_{Q,R_{D_1D_2}}<\beta_{Q,R_{D_1D_2}}$. In addition to that, we assume
\begin{equation}\label{e220}
\hbox{deg}(R_{\ell_1\ell_2})\le \hbox{deg}(R_{D_1D_2}),\quad 1\le \ell_1\le D_1-1, 1\le \ell_2\le D_2-1,\quad R_{D_1D_2}(im)\neq0,\quad m\in\R.
\end{equation}

We consider the main initial value problem under study 
\begin{multline}
Q(\partial_z)u(t_1,t_2,z,\epsilon)=\epsilon^{\Delta_{D_1D_2}}(t_1^{k_1+1}\partial_{t_1})^{\delta_{D_1}}(t_2^{k_2+1}\partial_{t_2})^{\delta_{D_2}}R_{D_1D_2}(\partial_z)u(t_1,t_2,z,\epsilon)\\
+\sum_{\substack{1\le\ell_1\le D_1-1\\1\le\ell_2\le D_2-1}}\epsilon^{\Delta_{\ell_1\ell_2}}(t_1^{k_1+1}\partial_{t_1})^{\delta_{\ell_1}}(t_2^{k_2+1}\partial_{t_2})^{\delta_{\ell_2}}c_{\ell_1\ell_2}(z,\epsilon)R_{\ell_1\ell_2}(\partial_z)u(t_1,t_2,z,\epsilon)+f(t_1,t_2,z,\epsilon),\label{epral}
\end{multline}
under given initial conditions $u(0,t_2,z,\epsilon)\equiv 0$, and $u(t_1,0,z,\epsilon)\equiv 0$. 

Let $\epsilon_0>0$. For every $0\le \ell_1\le D_1-1$ and $0\le \ell_2\le D_2-1$, the functions $c_{\ell_1\ell_2}(z,\epsilon)$ are holomorphic on $H_{\beta}\times D(0,\epsilon_0)$. They are defined by
$$c_{\ell_1\ell_2}(z,\epsilon)=\frac{1}{(2\pi)^{1/2}}\int_{-\infty}^{\infty}C_{\ell_1\ell_2}(m,\epsilon)e^{izm}dm,$$
where $m\mapsto C_{\ell_1\ell_2}(m,\epsilon)\in E_{(\beta,\mu)}$ and satisfies uniform bounds with respect to $\epsilon\in D(0,\epsilon_0)$. More precisely, there exists $\mathcal{C}_{\ell_1\ell_2}>0$ such that
\begin{equation}\label{e286v2}
\sup_{\epsilon\in D(0,\epsilon_0)}\left\|C_{\ell_1\ell_2}(m,\epsilon)\right\|_{(\beta,\mu)}\le \mathcal{C}_{\ell_1\ell_2}.
\end{equation}

The function $f$ is constructed as follows. Assume that $\psi:\C\times \R\times D(0,\epsilon_0)\to\C$ represents an entire function with respect to the first variable, continuous in $\R$ with respect to the second one, and holomorphic with respect to the third variable on the disc $D(0,\epsilon_0)$. Moreover, there exist $C_{\psi},\beta,\mu,\nu\in\R$, with $C_{\psi,}\beta,\nu>0$ and 
$$\mu>1+\hbox{deg}(R_{\ell_1\ell_2}),\quad 1\le \ell_1\le D_1-1,1\le \ell_2\le D_2-1,$$
such that the next bound holds
\begin{equation}\label{e319}
|\psi(\tau,m,\epsilon)|\le \frac{C_{\psi}}{(1+|m|)^{\mu}}e^{-\beta|m|}\exp\left(\nu|\tau|^{k'}\right)|\tau|,
\end{equation}
for all $(\tau,m,\epsilon)\in \C\times \R\times D(0,\epsilon_0)$, and some $k_1<k'<k_2$. In this situation, it is straight to check that the function 
\begin{equation}\label{e231}
F(T_1,T_2,m,\epsilon):=\int_{L_\gamma}\psi(u,m,\epsilon)\exp\left(-\left(\frac{u}{T_1}\right)^{k_1}-\left(\frac{u}{T_2}\right)^{k_2}\right)\frac{du}{u},
\end{equation}
where $L_\gamma=[0,\infty)e^{\gamma \sqrt{-1}}$ can spin around the origin in order to guarantee that $F$ is a holomorphic function on $\C^2$ with respect to $(T_1,T_2)$ by analytic continuation. We define
\begin{equation}\label{e233}
f(t_1,t_2,z,\epsilon)=\mathcal{F}^{-1}\left(m\mapsto F(\epsilon^{\lambda_1}t_1,\epsilon^{\lambda_1}t_2,m,\epsilon)\right)(z).
\end{equation}
Therefore, $f$ is a holomorphic function in $\C^2\times H_{\beta'}\times D(0,\epsilon_0)$, for every $0<\beta'<\beta$. The reason why the forcing term is built in such a restrictive manner will be put into light later on in Section~\ref{subsec2}.

In this framework, we search for solutions of (\ref{epral}) of the form
\begin{equation}\label{e186}
u(t_1,t_2,z,\epsilon)=\mathcal{F}^{-1}(m\mapsto U(\epsilon^{\lambda_1}t_1,\epsilon^{\lambda_2}t_2,m,\epsilon))(z).
\end{equation}
In view of (\ref{e193}), the properties of Fourier inverse transform, and the definition of $f$ and $c_{\ell_1\ell_2}$, we get that the expression $U(T_1,T_2,m,\epsilon)$ turns out to be a solution of
\begin{multline}
Q(im)U(T_1,T_2,m,\epsilon)=(T_1^{k_1+1}\partial_{T_1})^{\delta_{D_1}}(T_2^{k_2+1}\partial_{T_2})^{\delta_{D_2}}R_{D_1D_2}(im)U(T_1,T_2,m,\epsilon)\\
+\sum_{\substack{1\le\ell_1\le D_1-1\\1\le\ell_2\le D_2-1}}\epsilon^{\Delta_{\ell_1\ell_2}-\lambda_1 k_1\delta_{\ell_1}-\lambda_2 k_2\delta_{\ell_2}}(T_1^{k_1+1}\partial_{T_1})^{\delta_{\ell_1}}(T_2^{k_2+1}\partial_{T_2})^{\delta_{\ell_2}}\\
\times \frac{1}{(2\pi)^{1/2}}\int_{-\infty}^{\infty}C_{\ell_1\ell_2}(m-m_1,\epsilon)R_{\ell_1\ell_2}(im_1)U(T_1,T_2,m_1,\epsilon)dm_1+F(T_1,T_2,m,\epsilon).\label{epral1}
\end{multline}

\subsection{A first approach}\label{secfirstapproach}

As a first approach, one is tempted to follow techniques used in the previous works of the authors dealing with singularly perturbed partial differential equations in two complex time variables. On the one hand, the family of equations studied in~\cite{family1} shows a symmetric role of the time variables in the equation. Although this is the case for (\ref{epral}), in that previous study it holds that the principal part of any of the equations in the family is factorizable as a product of two operators which split the dependence on the time variables. For this reason, that procedure is no longer valid in the present framework. On the other hand, the study made in \cite{family2} does not fit the main problem under study, as it can be deduced from the forthcoming argument.

Departing from the auxiliary equation (\ref{epral1}), we proceed to apply $m_{k_1}-$Borel transformation with respect to $T_1$ and $m_{k_2}-$Borel transformation with respect to $T_2$. Then, it holds that
$$\omega(\tau_1,\tau_2,m,\epsilon):=\mathcal{B}_{m_{k_2},T_2}\mathcal{B}_{m_{k_1},T_1}(U(T_1,T_2,m,\epsilon))$$
solves an equation of the form
\begin{multline}
\left(Q(im)-(k_1\tau_1^{k_1})^{\delta_{D_1}}(k_2\tau_2^{k_2})^{\delta_{D_2}}R_{D_1D_2}(im)\right)\omega(\tau_1,\tau_2,m,\epsilon)\\
=\sum_{\substack{1\le\ell_1\le D_1-1\\1\le\ell_2\le D_2-1}}\epsilon^{\Delta_{\ell_1\ell_2}-\lambda_1 k_1\delta_{\ell_1}-\lambda_2 k_2\delta_{\ell_2}}(k_1\tau_1^{k_1})^{\delta_{\ell_1}}(k_2\tau_2^{k_2})^{\delta_{\ell_2}}\\
\frac{1}{(2\pi)^{1/2}}\int_{-\infty}^{\infty}C_{\ell_1\ell_2}(m-m_1,\epsilon)R_{\ell_1\ell_2}(im_1)\omega(\tau_1,\tau_2,m_1,\epsilon)dm_1+\tilde{F}(\tau_1,\tau_2,m,\epsilon),\label{epral2}
\end{multline}
where
$$\tilde{F}(\tau_1,\tau_2,m,\epsilon):=\mathcal{B}_{m_{k_1},T_1}\mathcal{B}_{m_{k_2},T_2}F(T_1,T_2,m,\epsilon)$$
is a holomorphic function on $\C^2$ w.r.t. the first variables, on $D(0,\epsilon_0)$ w.r.t. the fourth, and continuous on $\R$ w.r.t. the second variable, in view of (\ref{e231}).
     
Let 
\begin{equation}\label{e234}
P_m(\tau_1,\tau_2):=Q(im)-(k_1\tau_1^{k_1})^{\delta_{D_1}}(k_2\tau_2^{k_2})^{\delta_{D_2}}R_{D_1D_2}(im).
\end{equation}


\begin{prop}\label{prop0}
Under the previous construction leading to (\ref{epral2}), it holds that the possible actual holomorphic solutions $\omega(\tau_1,\tau_2,m,\epsilon)$ of (\ref{epral2}) can not be defined on any set of the form $(S_1\cup D(0,\rho))\times S_2$ (resp. $S_1\times(S_2\cup D(0,\rho))$) with respect to $(\tau_1,\tau_2)$, for any $\rho>0$ and any infinite sectors $S_1,S_2$ with vertex at the origin in $\C$.  
\end{prop}
\begin{proof}
Due to the equation (\ref{epral2}) exhibits a symmetric behavior with respect to $\tau_1$ and $\tau_2$, we only give details on the first of the previous statements, whilst the second follows from a symmetric argument. 

Fix $m\in\R$ and let $\rho_0>0$. We consider $\tau_2\in S_2$ such that 
$$|\tau_2|\ge\left(\frac{r^2_{Q,R_{D_1D_2}}}{\left(\rho_0/2\right)^{k_1\delta_{D_1}}}\frac{1}{k_1^{\delta_{D_1}}k_2^{\delta_{D_2}}}\right)^{\frac{1}{k_2\delta_{D_2}}}.$$
We derive that any $\tau_1\in\C$ such that $P_m(\tau_1,\tau_2)=0$ would satisfy that 
$$\tau_1^{k_1\delta_{D_1}}\tau_2^{k_2\delta_{D_2}}=\frac{Q(im)}{R_{D_1D_2}(im)}\frac{1}{k_1^{\delta_{D_1}}k_2^{\delta_{D_2}}},$$
which entails
$$|\tau_1|^{k_1\delta_{D_1}}\le\frac{1}{|\tau_2|^{k_2\delta_{D_2}}}r^2_{Q,R_{D_1D_2}}\frac{1}{k_1^{\delta_{D_1}}k_2^{\delta_{D_2}}}\le\left(\frac{\rho_0}{2}\right)^{k_1\delta_{D_1}}.$$
It follows that all the $k_1\delta_{D_1}$ roots of $\tau_1\mapsto P_m(\tau_1,\tau_2)$ for such choice of $\tau_2$ belong to the disc $D(0,\rho_0)$. The limit $\rho_0\to 0$ concludes the result. 
\end{proof}

As a matter of fact, a small divisor phenomena is observed, which does not allow a summability procedure. Moreover, the possible actual holomorphic solutions $\omega(\tau_1,\tau_2,m,\epsilon)$ of (\ref{epral2}) are only expected to be well defined and holomorphic on products of sectors with infinite radius, say $S_1\times S_2$. This construction does not allow to use the procedure applied neither in~\cite{family1} nor~\cite{family2}, in order to analyze the asymptotic properties of the solutions with respect to the small perturbation parameter $\epsilon$ (see the introduction of this work for further details).


\subsection{Second approach}\label{subsec2}

In view of the failure of the approaches in~\cite{family1,family2} (see Section~\ref{secfirstapproach}), we need to adopt another perspective. 


We search for solutions of (\ref{epral1}) of the special form
\begin{align}
U_{d}(T_1,T_2,m,\epsilon)&=\int_{L_d}\omega(u,m,\epsilon)\exp\left(-\left(\frac{u}{T_1}\right)^{k_1}-\left(\frac{u}{T_2}\right)^{k_2}\right)\frac{du}{u}\nonumber\\
&=\int_{L_d}\omega(u,m,\epsilon)\Omega(u,T_1,T_2)\frac{du}{u},\label{e285}
\end{align}
where $L_d$ stands for a half line departing from the origin and with bisecting direction of argument given by $d\in\R$, for some $d\in\R$ to be determined. If one pursues the construction of the first approach, one can construct genuine solutions of (\ref{epral1}) as a double Laplace
transform
$$ U_{d_{1},d_{2}}(T_{1},T_{2},m,\epsilon) = \frac{k_{1}k_{2}}{(2\pi)^{1/2}} \int_{L_{d_1}} \int_{L_{d_2}}
w(u_{1},u_{2},m,\epsilon) \exp( -(\frac{u_1}{T_1})^{k_1} ) \exp( -(\frac{u_2}{T_2})^{k_2} ) \frac{du_{1}}{u_1} \frac{du_{2}}{u_2}$$
along well chosen halflines $L_{d_j} \subset S_{j}$, $j=1,2$,
as it was the case in our previous studies~\cite{family1,family2}. Our idea consists on merging this double integral along the product
$L_{d_1} \times L_{d_2}$ into a simple integral along a halfline $L_{d} \subset \mathbb{C}$. Geometrically, it consists in a projection
on the diagonal part $(u,u) \in \mathbb{C}^{2}$, $u \in \mathbb{C}$ of the space $\mathbb{C}^{2}$. The advantage is that in the related
problem associated to the Borel map $w(u,m,\epsilon)$ (see (\ref{epral3})) involves now $P_{m}(\tau,\tau)$ as a denominator which is this time
well defined on a full neighborhood of 0 w.r.t $\tau$ and analytically continuable along unbounded sectors $S_{d}$ with suitable
directions $d \in \mathbb{R}$. As a result, the asymptotic analysis in $\epsilon$ becomes a tractable task. The drawback of this
approach concerns the class of equations we are able to handle which is reduced compared to our previous studies and contains linear
PDEs with special time reliance.

\vspace{0.3cm}

\noindent\textbf{Remark:}\label{aaa} Observe that in the case $k_1=k_2$, the function $U_d$ turns out to be a Laplace transform of order $k_1$ in the meromorphic function $$\Omega(T_1,T_2)=\frac{1}{\left(\frac{1}{T_1}\right)^{k_1}+\left(\frac{1}{T_2}\right)^{k_1}},$$
near $0\in\C^2$. That situation is directly linked to a summability procedure with respect to a germ of function in $\C^2$, as described in~\cite{mosch18}. 
 However, in our situation, the function $\Omega$ is meromorphic near 0, not analytic.

\vspace{0.3cm}

Under the hypothesis that the solution of (\ref{epral1}) is of the form (\ref{e285}), then $\omega$ in (\ref{e285}) solves the problem
\begin{multline}
\left(Q(im)-(k_1\tau^{k_1})^{\delta_{D_1}}(k_2\tau^{k_2})^{\delta_{D_2}}R_{D_1D_2}(im)\right)\omega(\tau,m,\epsilon)\\
=\sum_{\substack{1\le\ell_1\le D_1-1\\1\le\ell_2\le D_2-1}}\epsilon^{\Delta_{\ell_1\ell_2}-\lambda_1 k_1\delta_{\ell_1}-\lambda_2 k_2\delta_{\ell_2}}(k_1\tau^{k_1})^{\delta_{\ell_1}}(k_2\tau^{k_2})^{\delta_{\ell_2}}\\
\times \frac{1}{(2\pi)^{1/2}}\int_{-\infty}^{\infty}C_{\ell_1\ell_2}(m-m_1,\epsilon)R_{\ell_1\ell_2}(im_1)\omega(\tau,m_1,\epsilon)dm_1+\psi(\tau,m,\epsilon).\label{epral3}
\end{multline}

We substitute (\ref{epral2}) by (\ref{epral3}), as an auxiliary problem in order to solve the main equation. In the next section, we study some spaces of functions which are involved in the construction of the solution of (\ref{epral3}).

\subsection{Banach spaces of functions with exponential growth and decay}\label{subsec3}

The Banach spaces described in this section are modified versions of those appearing in~\cite{lama}. We omit the details which can be derived directly from that work, and the variations of that norm, stated in~\cite{family1,family2}.

We fix positive real numbers $\beta,\mu,\nu$, with $\mu>1$, and an integer $k'>0$. Let $\rho>0$ and $S_d$ be an open and unbounded sector with bisecting direction $d\in\R$. We denote $\overline{D}(0,\rho)$ the closed disc centered at 0 and positive radius $\rho$.

\begin{defin}
We write $\hbox{Exp}^d_{(\nu,\beta,\mu,k')}$ for the vector space of continuous complex value functions $(\tau,m)\mapsto h(\tau,m)$, defined on $(S_d\cup \overline{D}(0,\rho))\times \R$, holomorphic with respect to $\tau$ on $S_d\cup D(0,\rho)$ such that
$$\left\|h(\tau,m)\right\|_{(\nu,\beta,\mu,k')}:=\sup_{\tau\in S_d\cup\overline{D}(0,\rho),m\in\R}(1+|m|)^{\mu}e^{\beta|m|}\exp\left(-\nu|\tau|^{k'}\right)\frac{1}{|\tau|}|h(\tau,m)|<\infty.$$
The normed space $(\hbox{Exp}^d_{(\nu,\beta,\mu,k')},\left\|\cdot\right\|_{(\nu,\beta,\mu,k')})$ is a Banach space. 
\end{defin}

The next result is straighforward from the definition of the norm $\left\|\cdot\right\|_{(\nu,\beta,\mu,k')}$.

\begin{lemma}\label{lema1}
Let $(\tau,m)\mapsto a(\tau,m)$ be a bounded continuous function of $\overline{D}(0,\rho)\cup S_d)\times \R$, holomorphic with respect to $\tau$ on $D(0,\rho)\cup S_d$. Then, for every $h\in \hbox{Exp}^d_{(\nu,\beta,\mu,k')}$ it holds that $a(\tau,m)h(\tau,m)\in \hbox{Exp}^d_{(\nu,\beta,\mu,k')}$, and
$$\left\|a(\tau,m)h(\tau,m)\right\|_{(\nu,\beta,\mu,k')}\le \left(\sup_{(\tau,m)\in \overline{D}(0,\rho)\cup S_d)\times \R}|a(\tau,m)|\right)\left\|h(\tau,m)\right\|_{(\nu,\beta,\mu,k')}.$$
\end{lemma}

The proof of the next result follows analogous arguments as those in Proposition 2~\cite{lama19}, and refer to that work for a complete proof.

\begin{prop}\label{prop426}
Let $R_1R_2\in\C[X]$ such that
$$\hbox{deg}(R_1)\ge \hbox{deg}(R_2),\quad R_1(im)\neq0,\quad \mu>\hbox{deg}(R_2)+1.$$
Given $f\in E_{(\beta,\mu)}$ and $g\in \hbox{Exp}^d_{(\nu,\beta,\mu,k')}$, then it holds that the function
$$\Phi(\tau,m):=\frac{1}{R_1(im)}\int_{-\infty}^{\infty}f(m-m_1)R_2(im_1)g(\tau,m_1)dm_1,$$
is an element of $\hbox{Exp}^d_{(\nu,\beta,\mu,k')}$, and there exists $C_1>0$ such that 
$$\left\|\Phi(\tau,m)\right\|_{(\nu,\beta,\mu,k')}\le C_1\left\|f(m)\right\|_{(\beta,\mu)}\left\|g(\tau,m)\right\|_{(\nu,\beta,\mu,k')}.$$
\end{prop}

\subsection{Solution of an auxiliary equation}\label{subsec4}

At this point, we provide a brief summary on the strategy to trace. We continue with the approach described in Section~\ref{subsec2}, searching for solutions of (\ref{epral1}) in the form (\ref{e285}). In this section, we guarantee the existence of $\omega(\tau,m,\epsilon)$ by means of a fixed point argument in the Banach space of functions introduced in Section~\ref{subsec3}. 

At this point, we follow a similar guideline as the one initiated in our former study~\cite{lama}. We consider the next polynomial
\begin{equation}
P_{m}(\tau) = P_{m}(\tau,\tau) = Q(im) - k_{1}^{\delta_{D_1}}k_{2}^{\delta_{D_2}} \tau^{k_{1}\delta_{D_1} + k_{2}\delta_{D_2}}R_{D_{1}D_{2}}(im) \label{defin_Pmtau}
\end{equation}
In the following, we need lower bounds of the expression $P_{m}(\tau)$ with respect to $m$ and $\tau$.
In order to achieve this goal, we can factorize the polynomial w.r.t $\tau$, namely
\begin{equation}
P_{m}(\tau) = -k_{1}^{\delta_{D_1}}k_{2}^{\delta_{D_2}}R_{D_{1}D_{2}}(im)
\prod_{l=0}^{k_{1}\delta_{D_1}+k_{2}\delta_{D_2}-1} (\tau - q_{l}(m)) \label{factor_Pm}
\end{equation}
where its roots $q_{l}(m)$ can be displayed explicitely as
\begin{multline*}
q_{l}(m) =
\left( \frac{|Q(im)|}{|R_{D_{1}D_{2}}(im)|k_{1}^{\delta_{D_1}}k_{2}^{\delta_{D_2}}}
\right)^{\frac{1}{k_{1}\delta_{D_1} + k_{2}\delta_{D_2}}}\\
\times \exp
\left( \frac{\sqrt{-1}}{k_{1}\delta_{D_1}+k_{2}\delta_{D_2}}
( \mathrm{arg}( \frac{Q(im)}{R_{D_{1}D_{2}}(im)k_{1}^{\delta_{D_1}}k_{2}^{\delta_{D_2}}})
+ 2\pi l ) \right)
\end{multline*}
for all $0 \leq l \leq k_{1}\delta_{D_1}+k_{2}\delta_{D_2}-1$, for all $m \in \mathbb{R}$.

We set an unbounded sector $S_{d}$ centered at 0, a small disc $D(0,\rho)$ and we adjust the sector $A_{Q,R_{D_{1}D_{2}}}$ in a way that the next
condition hold : a constant $\mathfrak{m}>0$ can be chosen with
\begin{equation}
|\tau - q_{l}(m)| \geq \mathfrak{m}(1 + |\tau|) \label{dist_u_qlm}
\end{equation}
for all $0 \leq l \leq k_{1}\delta_{D_1} + k_{2}\delta_{D_2}-1$, all $m \in \mathbb{R}$, provided that $\tau \in S_{d} \cup D(0,\rho)$. Indeed, the inclusion
(\ref{e203}) implies in particular that all the roots $q_{l}(m)$, $0 \leq l \leq k_{1}\delta_{D_1}+k_{2}\delta_{D_2}-1$ remain apart of some
neighborhood of the origin, i.e satisfy $|q_{l}(m)| \geq 2\rho$ for an appropriate choice of $\rho>0$. Furthermore,
when the aperture of $A_{Q,R_{D_{1}D_{2}}}$ is taken close enough to 0, all these roots $q_{l}(m)$ stay inside a union
$\mathcal{U}$ of unbounded sectors centered at 0 that do not cover a full neighborhood of 0 in $\mathbb{C}^{\ast}$. We assign a sector $S_{d}$ with
$$ S_{d} \cap \mathcal{U} = \emptyset $$
By construction, the quotients $q_{l}(m)/\tau$ fall apart some small disc centered at 1 in $\mathbb{C}$ for all
$\tau \in S_{d}$, $m \in \mathbb{R}$, $0 \leq l \leq k_{1}\delta_{D_1} + k_{2}\delta_{D_2}-1$. Then, (\ref{dist_u_qlm}) follows.

We are now ready to supply lower bounds for $P_{m}(\tau)$.
\begin{lemma}
A constant $C_{P}>0$ (depending on $k_{1},k_{2},\delta_{D_1},\delta_{D_2},\mathfrak{m}$) can be found with
\begin{equation}
|P_{m}(\tau)| \geq C_{P}|R_{D_{1}D_{2}}(im)|(1 + |\tau|)^{k_{1}\delta_{D_1}+k_{2}\delta_{D_2}} \label{low_bds_Pmu}
\end{equation}
for all $\tau \in S_{d} \cup D(0,\rho)$, all $m \in \mathbb{R}$.
\end{lemma}
\begin{proof}
Departing from the factorization (\ref{factor_Pm}), the lower bounds (\ref{dist_u_qlm}) entail
$$ |P_{m}(\tau)| \geq k_{1}^{\delta_{D_1}}k_{2}^{\delta_{D_2}} \mathfrak{m}^{k_{1}\delta_{D_1}+k_{2}\delta_{D_2}}
|R_{D_{1}D_{2}}(im)| (1 + |\tau|)^{k_{1}\delta_{D_1}+k_{2}\delta_{D_2}} $$
for all $\tau \in S_{d} \cup D(0,\rho)$.
\end{proof}

\begin{lemma}\label{lema2}
Assume the conditions (\ref{e193})-(\ref{e220}) hold on the elements involved in the problem (\ref{epral}), with forcing term $f$ determined by the construction and conditions in (\ref{e319})-(\ref{e233}). 
We single out a sector $S_d$ that fulfills the constraints from the construction above.

Then, there exist $\epsilon_0,\varpi>0$ and $\xi_{\psi}>0$ (depending on $k_1,k_2,\delta_{D_1},\delta_{D_2}, Q, R_{D_1D_2}$), such that if $C_{\psi}\le\xi_{\psi}$, then for every $\epsilon\in D(0,\epsilon_0)$ the map $\mathcal{H}_{\epsilon}$ defined by
\begin{multline}
\mathcal{H}_{\epsilon}(\omega(\tau,m)):=\frac{1}{P_m(\tau)}\left(\sum_{\substack{1\le\ell_1\le D_1-1\\1\le\ell_2\le D_2-1}}\epsilon^{\Delta_{\ell_1\ell_2}-\lambda_1 k_1\delta_{\ell_1}-\lambda_2 k_2\delta_{\ell_2}}(k_1\tau^{k_1})^{\delta_{\ell_1}}(k_2\tau^{k_2})^{\delta_{\ell_2}}\right.\\
\left.\times\frac{1}{(2\pi)^{1/2}}\int_{-\infty}^{\infty}C_{\ell_1\ell_2}(m-m_1,\epsilon)R_{\ell_1\ell_2}(im_1)\omega(\tau,m_1)dm_1\right)+\frac{1}{P_m(\tau)}\psi(\tau,m,\epsilon)\label{e346}
\end{multline}
is such that:
\begin{itemize}
\item[i)] $\mathcal{H}_\epsilon(\overline{B}(0,\varpi))\subseteq B(0,\varpi)$, where $\overline{B}(0,\varpi)$ is the closed disc of radius $\varpi>0$, in $\hbox{Exp}^d_{(\nu,\beta,\mu,k')}$.
\item[ii)] We have
$$\left\|\mathcal{H}_\epsilon(\omega_1)-\mathcal{H}_\epsilon(\omega_2)\right\|_{(\nu,\beta,\mu,k')}\le\frac{1}{2}\left\|\omega_1-\omega_2\right\|_{(\nu,\beta,\mu,k')},$$
for all $\omega_1,\omega_2\in\overline{B}(0,\varpi)\subseteq \hbox{Exp}^d_{(\nu,\beta,\mu,k')}$.
\end{itemize}
\end{lemma}
\begin{proof}
Take $\omega\in \hbox{Exp}^d_{(\nu,\beta,\mu,k')}$, and $\epsilon\in D(0,\epsilon_0)$. Let $1\le \ell_1\le D_1-1$ and $1\le \ell_2\le D_2-1$. Bearing in mind Lemma~\ref{lema1}, Proposition~\ref{prop426}, from (\ref{e220}), (\ref{e286v2}) and (\ref{low_bds_Pmu}), we have
\begin{multline}
\left\|\epsilon^{\Delta_{\ell_1\ell_2}-\lambda_1 k_1\delta_{\ell_1}-\lambda_2 k_2\delta_{\ell_2}}(k_1\tau^{k_1})^{\delta_{\ell_1}}(k_2\tau^{k_2})^{\delta_{\ell_2}}\right.\\
\hfill\left.\times\frac{1}{(2\pi)^{1/2}}\int_{-\infty}^{\infty}\frac{C_{\ell_1\ell_2}(m-m_1,\epsilon)R_{\ell_1\ell_2}(im_1)}{P_m(\tau)}\omega(\tau,m_1)dm_1\right\|_{(\nu,\beta,\mu,k')}\\
\le C(\epsilon_0)\sup_{\tau\in S_d\cup\overline{D}(0,\rho),m\in\R}\frac{|\tau|^{k_1\delta_{\ell_1}+k_2\delta_{\ell_2}}}{(1+|\tau|)^{k_1\delta_{D_1}+k_2\delta_{D_2}}}\hfill\\
\hfill\times \left\|\frac{1}{R_{D_1D_2}(im)}\int_{-\infty}^{\infty} C_{\ell_1\ell_2}(m-m_1,\epsilon)R_{\ell_1\ell_2}(im)\omega(\tau,m_1)dm_1\right\|_{(\nu,\beta,\mu,k')}\\
\le C(\epsilon_0)C_1C_2\mathcal{C}_{\ell_1\ell_2}\left\|\omega(\tau,\epsilon)\right\|_{(\nu,\beta,\mu,k')}\le C(\epsilon_0)C_1C_2\mathcal{C}_{\ell_1\ell_2}\varpi,\label{e376}
\end{multline}
where 
$$C(\epsilon_0)=\sup_{\substack{1\le\ell_1\le D_1-1\\1\le\ell_2\le D_2-1}}\frac{\epsilon_0^{\Delta_{\ell_1\ell_2}-\lambda_1 k_1\delta_{\ell_1}-\lambda_2 k_2\delta_{\ell_2}}k_1^{\delta_{\ell_1}-\delta_{D_1}}k_2^{\delta_{\ell_2}-\delta_{D_2}}}{(2\pi)^{1/2}C_P\mathfrak{m}^{k_1\delta_{D_1}+k_2\delta_{D_2}}}>0,$$
is such that $C(\epsilon_0)\to0$ for $\epsilon_0\to0$, and 
$$C_2=\sup_{\tau\in S_d\cup\overline{D}(0,\rho),m\in\R}\frac{|\tau|^{k_1\delta_{\ell_1}+k_2\delta_{\ell_2}}}{(1+|\tau|)^{k_1\delta_{D_1}+k_2\delta_{D_2}}}.$$

On the other hand, in view of (\ref{e220}), (\ref{e319}) and (\ref{low_bds_Pmu}), we derive that
\begin{multline}
$$\left\|\frac{1}{P_m(\tau)}\psi(\tau,m,\epsilon)\right\|_{(\nu,\beta,\mu,k')}\le \sup_{\substack{\tau\in S_d\cup \overline{D}(0,\rho)\\m\in\R}}\frac{C_\psi}{C_p|R_{D_1D_2}(im)|(1+|\tau|)^{k_1\delta_{D_1}+k_2\delta_{D_2}}}\le C_3\xi_{\psi},$$
\end{multline}
for some $C_3>0$.


Let $\varpi,\xi_\psi,\epsilon_0>0$ such that $C(\epsilon_0)C_1C_2(\sum_{\substack{1\le\ell_1\le D_1-1\\1\le\ell_2\le D_2-1}}\mathcal{C}_{\ell_1\ell_2})\varpi+\xi_\psi C_3\le\varpi$. Under this choice, we get that $\mathcal{H}_\epsilon$ is such that $\mathcal{H}_\epsilon(\overline{B}(0,\varpi))\subseteq\overline{B}(0,\varpi)$. For the second part of the proof, we choose $\omega_1,\omega_2\in \hbox{Exp}^d_{(\nu,\beta,\mu,k')}$, with $\left\|\omega_j\right\|_{(\nu,\beta,\mu,k')}\le\varpi$. Analogous estimates as in (\ref{e376}) yield
\begin{multline}\left\|\epsilon^{\Delta_{\ell_1\ell_2}-\lambda_1 k_1\delta_{\ell_1}-\lambda_2 k_2\delta_{\ell_2}}(k_1\tau^{k_1})^{\delta_{\ell_1}}(k_2\tau^{k_2})^{\delta_{\ell_2}}\frac{R_{\ell_1\ell_2}(im)}{P_m(\tau)}(\omega_1(\tau,m)-\omega_2(\tau,m))\right\|_{(\nu,\beta,\mu,k')}\\
\le C(\epsilon_0)C_1C_2\mathcal{C}_{\ell_1\ell_2}\left\|\omega_1(\tau,m)-\omega_2(\tau,m)\right\|_{(\nu,\beta,\mu,k')},\label{e376b}
\end{multline}
and consequently 
$$ \left\|\mathcal{H}_\epsilon(\omega_1)-\mathcal{H}_\epsilon(\omega_2)\right\|_{(\nu,\beta,\mu,k')}\le C(\epsilon_0)C_1C_2\left(\sum_{\substack{1\le\ell_1\le D_1-1\\1\le\ell_2\le D_2-1}}\mathcal{C}_{\ell_1\ell_2}\right)\left\|\omega_1(\tau,m)-\omega_2(\tau,m)\right\|_{(\nu,\beta,\mu,k')}.$$
Let $\epsilon_0>0$ be such that $C(\epsilon_0)C_1C_2\left(\sum_{1\le\ell_1\le D_1-1,1\le\ell_2\le D_2-1}\mathcal{C}_{\ell_1\ell_2}\right)\le1/2$. This entails that $\mathcal{H}_\epsilon$ is a contractive map in $\overline{B}(0,\varpi)\subseteq \hbox{Exp}^d_{(\nu,\beta,\mu,k')}$. 
\end{proof}

As a consequence of Lemma~\ref{lema2}, we achieve the next result.

\begin{prop}\label{prop1}
Assume the hypotheses of Lemma~\ref{lema2} hold. Let $\varpi>0$. Then, there exist $\epsilon_0>0$ and $\xi_\psi>0$, such that if $C_\psi<\xi_\psi$, then for every $\epsilon\in D(0,\epsilon_0)$, the equation (\ref{epral3}) admits a solution $\omega(\tau,m,\epsilon)\in \hbox{Exp}^d_{(\nu,\beta,\mu,k')}$, with $\left\|\omega(\tau,m,\epsilon)\right\|_{(\nu,\beta,\mu,k')}\le\varpi$.
\end{prop}

\begin{proof}
By the classical fixed point theorem in Banach spaces, Lemma~\ref{lema2} guarantees the existence of a function $\omega(\tau,m,\epsilon)\in B(0,\varpi)\subseteq \hbox{Exp}^d_{(\nu,\beta,\mu,k')}$ such that $\mathcal{H}_{\epsilon}(\omega(\tau,m,\epsilon))=\omega(\tau,m,\epsilon)$, for every $\epsilon\in D(0,\epsilon_0)$. Moreover, the function $\omega(\tau,m,\epsilon)$ depends holomorphically on $\epsilon\in D(0,\epsilon_0)$. One can directly check that this fixed point is a solution of the problem (\ref{epral3}).
\end{proof}

\section{Actual solutions of the auxiliar problem (\ref{epral3})}

Let $\mathcal{E}$ be a finite sector of $\C^{\star}$ with vertex at the origin. Let $d\in\R$ the bisecting direction of an infinite sector $S_d$ satisfying the hypotheses of Proposition~\ref{prop1}. Let $\omega_{d}$ be the solution of (\ref{epral3}) constructed in Proposition~\ref{prop1}. Let $\tilde{\mathcal{T}}_1$ be a \textit{bounded} sector with vertex at the origin, and let $\tilde{\mathcal{T}}_2$ be an \textit{unbounded} sector with vertex at the origin such that
\begin{equation}\label{e4260}
k_j\xi-k_j\hbox{arg}(T_j)\in\left(-\frac{\pi}{2}+\delta,\frac{\pi}{2}-\delta\right),
 \end{equation}
for all $T_j\in \tilde{\mathcal{T}}_j$, some $\xi\in\R$ (which might depend on $T_1$ and $T_2$) such that $e^{\sqrt{-1}\xi}\in S_{d}$, all $\epsilon\in\mathcal{E}$ and some $\delta>0$; for $j=1,2$. 

In particular, observe there exists $\delta_1>0$ such that $\cos(k_1\xi-k_1\hbox{arg}(T_1))>\delta_1$, for every $T_1\in\tilde{\mathcal{T}}_1$, $\epsilon\in\mathcal{E}$; and there exists $\delta_2>0$ such that $\cos(k_2\xi-k_2\hbox{arg}(T_2))>\delta_2$, for every $T_2\in\tilde{\mathcal{T}}_2$, $\epsilon\in\mathcal{E}$.

The function $U_{\xi}$ defined in (\ref{e285}) turns out to be an actual solution of the auxiliary problem (\ref{epral3}) in the domain $\tilde{\mathcal{T}}_1\times \tilde{\mathcal{T}}_2\times \R\times \mathcal{E}$. Moreover, the next estimates hold
\begin{multline}
|U_{\xi}(T_1,T_2,m,\epsilon)|\le \frac{\varpi(1+|m|)^{-\mu}}{e^{\beta|m|}}\int_0^\infty e^{\nu r^{k'}}\\
\times \exp\left(-\frac{r^{k_1}}{|T_1|^{k_1}}\cos(k_1\xi-k_1\hbox{arg}(T_1))-\frac{r^{k_2}}{|T_2|^{k_2}}\cos(k_2\xi-k_2\hbox{arg}(T_2))\right)dr\\
\le \varpi (1+|m|)^{-\mu}e^{-\beta|m|}L(|T_1|,|T_2|),
\end{multline}
where 
$$L(|T_1|,|T_2|)=\int_0^{\infty}e^{\nu r^{k'}}\exp\left(-\frac{r^{k_1}}{|T_1|^{k_1}}\delta_1-\frac{r^{k_2}}{|T_2|^{k_2}}\delta_2\right)dr.$$

Due to $k'\in(k_1,k_2)$, the function $L(x,y)$ is well defined in $\{(x,y)\in\R^2: x\ge0,y\ge0\}$.

We write $L(|T_1|,|T_2|)=L_1(|T_1|,|T_2|)+L_2(|T_1|,|T_2|)$, with
$$L_1(|T_1|,|T_2|)=\int_0^{\rho}e^{\nu r^{k'}}e^{-\frac{r^{k_1}}{|T_1|^{k_1}}\delta_1-\frac{r^{k_2}}{|T_2|^{k_2}}\delta_2}dr,\qquad L_2(|T_1|,|T_2|)=\int_{\rho}^{\infty}e^{\nu r^{k'}}e^{-\frac{r^{k_1}}{|T_1|^{k_1}}\delta_1-\frac{r^{k_2}}{|T_2|^{k_2}}\delta_2}dr,$$
for some $\rho>0$.

The proof of the following technical lemma is left to the Appendix at the end of the work, in order not to interfere with the ongoing arguments.

\begin{lemma}\label{lema3}
The following statements hold:
\begin{itemize}
\item[1)] There exists $C_1>0$ such that $0< L_1(|T_1|,|T_2|)\le C_1$ for all $|T_1|,|T_2|\ge0$.
\item[2.a)] Let $\rho_1>0$. There exist large enough $\rho_2^{\infty}>0$ such that 
$$L_2(|T_1|,|T_2|)\le C_{2a}e^{-\frac{\rho^{k_1}}{|T_1|^{k_1}}\delta_1}|T_2|^{1+\frac{k_2}{k_2-k'}}\exp\left(\nu^{(1-k'/k_2)^{-1}}\left(\frac{1}{\delta_2}\right)^{k'/(k_2-k')}|T_2|^{k_2k'/(k_2-k')}\right),$$
for all $|T_1|<\rho_1$ and $|T_2|>\rho_2^\infty$, and some $C_{2a}>0$. 
\item[2.b)] Let $\rho_1,\rho_2>0$. Then, it holds that
$$L_2(|T_1|,|T_2|)\le C_{2b}\exp\left(-\frac{\rho^{k_1}}{|T_1|^{k_1}}\delta_1\right)\exp\left(-\frac{\rho^{k_2}}{2|T_2|^{k_2}}\delta_2\right),$$
for all $|T_1|<\rho_1$ and $|T_2|<\rho_2$, and some  $C_{2b}>0$. 
\end{itemize}
\end{lemma}

Following our new approach, we are able to construct global solutions $U_{\xi}(T_{1},T_{2},m,\epsilon)$ in the time
variable $T_{2}$ on an unbounded sector $\tilde{\mathcal{T}}_{2}$. This was not possible in our two previous studies~\cite{family1,family2}, where only
local in time solutions were built up. This feature allows us to study the asymptotic expansions w.r.t $\epsilon$ in two different
situations : when $T_{2}$ remains in a prescribed bounded domain (which is related to the forthcoming outer solution, constructed for the main equation~\ref{epral}) and when $T_{2}$
tends to $\infty$ in a related manner with $\epsilon$ (linked to the inner solutions of (\ref{epral}) that we plan to build in Section~\ref{secin}).

\section{Inner and outer solutions of the main problem}\label{sec4}

In this section, we preserve the conditions established in the statement of the main problem under study in Section~\ref{sec1}. More precisely, we assume the conditions (\ref{e193})-(\ref{e220}) hold on the elements involved in the problem (\ref{epral}), with forcing term $f$ determined by the construction and conditions in (\ref{e319})-(\ref{e233}).

 Our main aim is to construct solutions of the main problem (\ref{epral}) together with their asymptotic behavior in different situations. This will be done via (\ref{epral1}) and the shape (\ref{e285}), as stated in our second approach in Section~\ref{subsec2}.

We first specify some geometric constructions on the domain of definition of the solutions. First, we recall the definition of a good covering in $\C^{\star}$, and that of a good covering of prescribed opening.

\begin{defin}\label{defi4}
 Let $\iota_1,\iota_2\ge2$ be integer numbers. We consider two sets $(\mathcal{E}^0_{h_1})_{0\le h_1\le\iota_1-1}$ and $(\mathcal{E}^\infty_{h_2})_{0\le h_2\le\iota_2-1}$, where $\mathcal{E}^0_{h_1},\mathcal{E}^\infty_{h_2}\subseteq D(0,\epsilon_0)$ are open sectors with vertex at the origin which satisfy the following assumptions:
\begin{itemize}
\item[i)] $\mathcal{E}^0_{h_1}\cap\mathcal{E}^0_{h_1+1}\neq \emptyset$ for all $0\le h_1\le \iota_1-1$ (with $\mathcal{E}_{\iota_1}:=\mathcal{E}_0^0$), and $\mathcal{E}^\infty_{h_2}\cap\mathcal{E}^\infty_{h_2+1}\neq \emptyset$ for all $0\le h_2\le \iota_2-1$ (with $\mathcal{E}_{\iota_2}:=\mathcal{E}_0^\infty$)
\item[ii)] The intersection of three different elements of each family is empty.
\item[iii)] The union of the elements of each family covers a punctured disc centered at 0 in $\C$.
\item[iv)] The opening of $\mathcal{E}_{h_1}^0$ is larger than $\pi/(\lambda_2k_2)$ for all $0\le h_1\le \iota_1-1$.
\end{itemize}
Then we say that $(\mathcal{E}_{h_1}^{0})_{0\le h_1\le \iota_1-1}$ is a good covering of $\C^\star$ of opening $\pi/(\lambda_2 k_2)$, and $(\mathcal{E}_{h_2}^{\infty})_{0\le h_2\le \iota_2-1}$ is just called a good covering in $\C^\star$.
\end{defin}

\begin{defin}
Let $\mathcal{T}_1$ be a bounded sector with vertex at the origin and let $\mathcal{T}_2$ be an unbounded sector with vertex at the origin. Let $(\mathcal{E}_{h_1}^{0})_{0\le h_1\le \iota_1-1}$ be a good covering of $\C^\star$ of opening $\pi/(\lambda_2 k_2)$, and let $(\mathcal{E}_{h_2}^{\infty})_{0\le h_2\le \iota_2-1}$ be a good covering in $\C^\star$. For every $0\le h_1\le \iota_1-1$, let $S_{d_{h_1}}^0$ be an infinite sector of bisecting direction $d_{h_1}$ and let $S_{d_{h_2}}^{\infty}$ be an infinite sector of bisecting direction $d_{h_2}$, for all $0\le h_2\le \iota_2-1$. We say that the set $\{\mathcal{T}_1,\mathcal{T}_2,(\mathcal{E}_{h_1}^{0})_{0\le h_1\le \iota_1-1},(S_{h_1}^{0})_{0\le h_1\le \iota_1-1}\}$ is admissible if it holds that
\begin{equation}\label{e426}
k_j\xi_{h_1}-k_j\hbox{arg}(\epsilon^{\lambda_j}t_j)\in\left(-\frac{\pi}{2}+\delta,\frac{\pi}{2}-\delta\right),
\end{equation}
for all $t_j\in\mathcal{T}_j$, some $\xi_{h_1}\in\R$ (which might depend on $t_j$ and $\epsilon$) such that $e^{\sqrt{-1}\xi_{h_1}}\in S_{d_{h_1}}$, all $\epsilon\in\mathcal{E}^{0}_{h_1}$ and some $\delta>0$; for $j=1,2$.

The set $\{\mathcal{T}_1,\mathcal{T}_2,(\mathcal{E}_{h_2}^{\infty})_{0\le h_2\le \iota_2-1},(S_{h_2}^{0})_{0\le h_2\le \iota_2-1}\}$ is admissible if it holds that
\begin{equation}\label{e426b}
k_j\xi_{h_2}-k_j\hbox{arg}(\epsilon^{\lambda_j}t_j)\in\left(-\frac{\pi}{2}+\delta,\frac{\pi}{2}-\delta\right),
\end{equation}
for all $t_j\in\mathcal{T}_j$, some $\xi_{h_2}\in\R$ (which might depend on $t_j$ and $\epsilon$) such that $e^{\sqrt{-1}\xi_{h_2}}\in S_{d_{h_2}}$, all $\epsilon\in\mathcal{E}^{\infty}_{h_2}$ and some $\delta>0$; for $j=1,2$.
\end{defin}

\subsection{Construction and results on the inner solutions}\label{secin}

\begin{defin}\label{defi481}
Let $\mu_2>\lambda_2$ be an integer number in such a way that 
\begin{equation}\label{e491}
\lambda_1k_1>(\mu_2-\lambda_2)\left(\frac{1}{k'}-\frac{1}{k_2}\right)^{-1}.
\end{equation}
Let $\chi_2^{\infty}$ be the bounded domain 
$$\chi_2^{\infty}=\left\{x_2\in\C^{\star}:r_{2,\infty}<|x_2|<R_{2,\infty},\hbox{arg}(x_2)\in(\alpha_{2,\infty},\beta_{2,\infty})\right\},$$
for some real numbers $0<r_{2,\infty}<R_{2,\infty}$, and $\alpha_{2,\infty}<\beta_{2,\infty}$.

We assume that the good covering $(\mathcal{E}_{h_2}^{\infty})_{0\le h_2\le \iota_2-1}$ satisfies the next additional condition: for all $0\le h_2\le \iota_2-1$ we can choose $\theta_{h_2}\in\R$ (which depends on $\mathcal{E}^{\infty}_{h_2}$) such that for all $x_2\in\chi_2^{\infty}$ and $\epsilon\in \mathcal{E}^{\infty}_{h_2}$, the complex number
\begin{equation}\label{e484}
t_2=\frac{x_2}{\epsilon^{\mu_2}}e^{\theta_{h_2} \sqrt{-1}}
\end{equation}
belongs to $\mathcal{T}_2$ for all $\epsilon\in\mathcal{E}^{\infty}_{h_2}$. Then, we define the set 
$$\mathcal{T}_{2,\epsilon,\mu_2}:=\left\{\frac{x_2}{\epsilon^{\mu_2}}e^{\theta_{h_2}\sqrt{-1}}:x_2\in \chi_2^\infty\right\}.$$
\end{defin}

\vspace{0.3cm}

\noindent \textbf{Remark:} Observe that $\mathcal{T}_{2,\epsilon,\mu_2}\subseteq \mathcal{T}_2$ for every $\epsilon\in\mathcal{E}^{\infty}_{h_2}$ and $0\le h_2\le \iota_2-1$. In addition to this, observe that $\mathcal{T}_{2,\epsilon,\mu_2}$ is a bounded domain for every $\epsilon\in\mathcal{E}^{\infty}_{h_2}$ for $0\le h_2\le \iota_2-1$ with
$$\lim_{\epsilon\to0,\epsilon\in\mathcal{E}^{\infty}_{h_2}}\hbox{dist}(\mathcal{T}_{2,\epsilon,\mu_2},0)=\infty.$$

\vspace{0.3cm}


\begin{theo}\label{lema495}
Let $\{\mathcal{T}_1,\mathcal{T}_2,(\mathcal{E}_{h_2}^{\infty})_{0\le h_2\le \iota_2-1},(S_{h_2}^{\infty})_{0\le h_2\le \iota_2-1}\}$ be an admissible set. For every $0\le h_2\le\iota_2-1$ and $\epsilon\in\mathcal{E}^{\infty}_{h_2}$, the function
$$(t_1,t_2,z)\mapsto u_{d_{h_2}}(t_1,t_2,z,\epsilon),$$
where 
\begin{align}u_{d_{h_2}}(t_1,t_2,z,\epsilon)&:=\mathcal{F}^{-1}(m\mapsto U_{\xi_{h_2}}(\epsilon^{\lambda_1}t_1,\epsilon^{\lambda_2}t_2,m,\epsilon))(z),\nonumber\\ 
&=\frac{1}{(2\pi)^{1/2}}\int_{-\infty}^{\infty}\int_{L_{\xi_{h_2}}}\omega_{d_{h_2}}(u,m,\epsilon)\exp\left(-\left(\frac{u}{\epsilon^{\lambda_1}t_1}\right)^{k_1}-\left(\frac{u}{\epsilon^{\lambda_2}t_2}\right)^{k_2}\right)\exp(izm)\frac{du}{u}dm,\label{e472}
\end{align}
where $\omega_{d_{h_2}}(u,m,\epsilon)$ is constructed in Section~\ref{subsec4}, defines a bounded holomorphic function on $\mathcal{T}_1\times\mathcal{T}_{2,\epsilon,\mu_2}\times H_{\beta'}$ for all $0<\beta'<\beta$, which is an actual solution of (\ref{epral}), called an \textbf{\textit{inner solution}}. 
 Moreover, for every $\epsilon\in\mathcal{E}^{\infty}_{h_2}\cap\mathcal{E}^{\infty}_{h_2+1}$, there exist $\tilde{C},\tilde{D}>0$ such that 
\begin{multline}
\sup_{t_1\in\mathcal{T}_1,x_2\in \chi_2^{\infty},z\in H_{\beta'}}|u_{d_{h_2+1}}(t_1,x_2\epsilon^{-\mu_2}e^{\theta_{h_2}\sqrt{-1}},z,\epsilon)-u_{d_{h_2}}(t_1,x_2\epsilon^{-\mu_2}e^{\theta_{h_2}\sqrt{-1}},z,\epsilon)|\label{e495}\\
\le \tilde{C}\exp\left(\frac{\tilde{D}}{|\epsilon|^{\lambda_1k_1}}\right).
\end{multline}
\end{theo}
\begin{proof}
Let $0\le h_2\le \iota_2-1$, and $d_{h_2}$ be the value  $d$ determined in Proposition~\ref{prop1}. The function $\omega(\tau,m,\epsilon)\in\hbox{Exp}^{d_{h_2}}_{(\nu,\beta,\mu,k')}$ solves equation (\ref{epral3}). This entails that $U_{\xi_{h_2}}(\epsilon^{\lambda_1}t_1,\epsilon^{\lambda_2}t_2,m,\epsilon)$ is a solution of (\ref{epral1}), and its inverse Fourier transform with respect to $m$ turns out to be a solution of (\ref{epral}). The geometric construction of the domains involved guarantees that the map $(t_1,t_2,z)\mapsto u_{d_{h_2}}(t_1,t_2,z,\epsilon),$ represents a bounded holomorphic function defined in $\mathcal{T}_1\times\mathcal{T}_{2,\epsilon,\mu_2}\times H_{\beta'}$ for every $0<\beta'<\beta$, where $\epsilon$ belongs to $\mathcal{E}_{h_2}^{\infty}$.


Regarding the definition of $\tau\mapsto P_m(\tau)$ in (\ref{defin_Pmtau}) and (\ref{low_bds_Pmu}), we get that all the roots of $P_m(\tau)$ are at positive distance, say $\rho>0$, to the origin. This entails that the integration path defining the difference of two consecutive inner solutions can be deformed as follows.

Let $0\le h_2\le \iota_2-1$, fix $\epsilon\in\mathcal{E}^{\infty}_{h_2}\cap\mathcal{E}^{\infty}_{h_2+1}$, and write $t_2=x_2\epsilon^{-\mu_2}e^{\theta_{h_2}\sqrt{-1}}\in\mathcal{T}_{2,\epsilon,\mu_2}$, for some $x_2\in\chi_2^{\infty}$ and $\theta_{h_2}\in\R$, as described in Definition~\ref{defi481}. For every $t_1\in\mathcal{T}_1$, $z\in H_{\beta'}$ we get that
\begin{multline*}
u_{d_{h_2+1}}(t_1,t_2,z,\epsilon)-u_{d_{h_2}}(t_1,t_2,z,\epsilon)\\
=\frac{1}{(2\pi)^{1/2}}\int_{-\infty}^{\infty}\int_{L_{\xi_{h_2+1}}}\omega(u,m,\epsilon)\Omega(u,\epsilon^{\lambda_1}t_1,\epsilon^{\lambda_2}t_2)\frac{du}{u}e^{izm}dm\\
-\frac{1}{(2\pi)^{1/2}}\int_{-\infty}^{\infty}\int_{L_{\xi_{h_2}}}\omega(u,m,\epsilon)\Omega(u,\epsilon^{\lambda_1}t_1,\epsilon^{\lambda_2}t_2)\frac{du}{u}e^{izm}dm=E_1-E_2+E_3,
\end{multline*}
with $\Omega(u,T_1,T_2)$ defined in (\ref{e285}), and where
\begin{multline*}
E_1:=\frac{1}{(2\pi)^{1/2}}\int_{-\infty}^{\infty}\int_{L_{\xi_{h_2+1},\rho/2}}\omega(u,m,\epsilon)\Omega(u,\epsilon^{\lambda_1}t_1,\epsilon^{\lambda_2}t_2)\frac{du}{u}e^{izm}dm,\\
E_2:=\frac{1}{(2\pi)^{1/2}}\int_{-\infty}^{\infty}\int_{L_{\xi_{h_2},\rho/2}}\omega(u,m,\epsilon)\Omega(u,\epsilon^{\lambda_1}t_1,\epsilon^{\lambda_2}t_2)\frac{du}{u}e^{izm}dm,
\end{multline*}
and where $L_{\xi_{h_2+1},\rho/2}:=[\rho/2,+\infty)e^{\xi_{h_2+1}\sqrt{-1}}$, $L_{\xi_{h_2},\rho/2}:=[\rho/2,+\infty)e^{\xi_{h_2}\sqrt{-1}}$, and
$$E_3:= \frac{1}{(2\pi)^{1/2}}\int_{-\infty}^{\infty}\int_{C_{h_2,h_2+1,\rho/2}}\omega(u,m,\epsilon)\Omega(u,\epsilon^{\lambda_1}t_1,\epsilon^{\lambda_2}t_2)\frac{du}{u}e^{izm}dm,$$
where $C_{h_2,h_2+1,\rho/2}$ is the arc of circle departing from $\rho/2e^{\xi_{h_2+1}\sqrt{-1}}$, ending at $\rho/2e^{\xi_{h_2}\sqrt{-1}}$.
Bearing in mind that $\mu_2>\lambda_2$, (\ref{e484}) and Lemma~\ref{lema3}, one can apply 
2.a) in Lemma~\ref{lema3} to arrive at
\begin{multline}\label{e523}
|E_1|\le \int_{-\infty}^{\infty}\int_{\rho/2}^{\infty}\frac{\varpi}{(1+|m|)^{\mu}}e^{-\beta|m|}e^{\nu r^{k'}}\exp\left(-\left(\frac{r}{|\epsilon^{\lambda_1}t_1|}\right)^{k_1}\delta_1-\left(\frac{r}{|\epsilon^{\lambda_2}t_2|}\right)^{k_2}\delta_2\right)dre^{|m|\hbox{Im}(z)}dm\\
\le \tilde{C}_1 \int_{\rho/2}^{\infty}\exp\left(\nu r^{k'}\right)\exp\left(-\left(\frac{r}{|\epsilon^{\lambda_1}t_1|}\right)^{k_1}\delta_1-\left(\frac{r}{|\epsilon^{\lambda_2}t_2|}\right)^{k_2}\delta_2\right)dr\\
\le \tilde{C}_1 C_{2a}\exp\left(-\frac{(\rho/2)^{k_1}\delta_1}{|\epsilon^{\lambda_1}t_1|^{k_1}}\right)|\epsilon^{\lambda_2}t_2|^{1+\frac{k_2}{k_2-k'}}\exp\left(\nu^{(1-k'/k_2)^{-1}}\left(\frac{1}{\delta_2}\right)^{\frac{k'}{k_2-k'}}|\epsilon^{\lambda_2}t_2|^{\frac{k_2k'}{k_2-k'}}\right)\\
\le \tilde{C}_1 C_{2a}\exp\left(-\frac{(\rho/2)^{k_1}\delta_1}{|\epsilon^{\lambda_1}|^{k_1}C_{\mathcal{T},1}^{k_1}}\right)|\epsilon^{\lambda_2}|^{1+\frac{k_2}{k_2-k'}}C_{\chi_2^\infty}^{\frac{k_2k'}{k_2-k'}}|\epsilon|^{-\mu_2(1+\frac{k_2}{k_2-k'})}\\
\times\exp\left(\nu^{(1-k'/k_2)^{-1}}\left(\frac{1}{\delta_2}\right)^{\frac{k'}{k_2-k'}}|\epsilon^{\lambda_2}|^{\frac{k_2k'}{k_2-k'}}C_{\chi_2^\infty}^{1+\frac{k_2}{k_2-k'}}|\epsilon|^{-\mu_2\frac{k_2k'}{k_2-k'}}\right)
\end{multline}
for some positive constant $\tilde{C}_1,C_{\mathcal{T},1},C_{\chi_2^{\infty}}>0$. Taking into account (\ref{e491}), we derive that 
\begin{equation}\label{e546}
|E_1|\le \tilde{C}_2\exp\left(-\frac{\tilde{D}_2}{|\epsilon|^{\lambda_1 k_1}}\right),
\end{equation}
for some $\tilde{C}_2,\tilde{D}_2>0$.

An analogous upper bound can be associated to $|E_2|$. We finally consider $|E_3|$. Analogous estimates as for the previous yield
\begin{multline*}|E_3|\le \frac{1}{(2\pi)^{1/2}}\int_{-\infty}^{\infty}\int_{\xi_{h_2}}^{\xi_{h_2+1}}\frac{\varpi}{(1+|m|)^{\mu}}e^{-\beta|m|}e^{\nu(\rho/2)^{k'}}\\
\times\exp\left(-\frac{(\rho/2)^{k_1}\delta_1}{|\epsilon|^{\lambda_1 k_1}C_{\mathcal{T},1}}-\frac{(\rho/2)^{k_2}\delta_2|\epsilon|^{\mu_2k_2}}{|\epsilon|^{\lambda_2 k_2}C_{\chi_2^{\infty}}}\right)dr e^{|m|\hbox{Im}(z)}dm\le \tilde{C}_3 \exp\left(-\frac{(\rho/2)^{k_1}\delta_1}{|\epsilon|^{\lambda_1 k_1}C_{\mathcal{T},1}}\right),
\end{multline*}
for some $\tilde{C}_3,C_{\mathcal{T},1},C_{\chi_2^{\infty}}>0$. This entails 
\begin{equation}\label{e546b}
|E_3|\le \tilde{C}_4\exp\left(-\frac{\tilde{D}_4}{|\epsilon|^{\lambda_1 k_1}}\right),
\end{equation}
for some $\tilde{C}_4,\tilde{D}_4>0$. We conclude (\ref{e495}) from (\ref{e546}) and (\ref{e546b}).
\end{proof}

We now fix a different good covering to provide the asymptotic behavior of the outer solutions. Let $\iota_1\ge 2$ be an integer number. We fix a good covering in $\C^{\star}$, $(\mathcal{E}^0_{h_1})_{0\le h_1\le\iota_1-1}$ of opening $\frac{\pi}{\lambda_2 k_2}$. We observe that (\ref{e426}), (\ref{e426b}) is satisfied for $j=2$ under the second assumption in (\ref{e193}).

\subsection{Construction and results on the outer solutions}\label{secout}


\begin{theo}\label{teo3}
Let $\rho_2>0$ and $0\le h_1\le \iota_1-1$. Let $\{\mathcal{T}_1,\mathcal{T}_2,(\mathcal{E}_{h_1}^{0})_{0\le h_1\le \iota_1-1},(S_{h_1}^{0})_{0\le h_1\le \iota_1-1}\}$ be an admissible set. For every $0\le h_1\le \iota_1-1$ the function $(t_1,t_2,z,\epsilon)\mapsto u_{d_{h_1}}(t_1,t_2,z,\epsilon)$ defined by 
\begin{align}
u_{d_{h_1}}(t_1,t_2,z,\epsilon)&:=\mathcal{F}^{-1}(m\mapsto U_{\xi_{h_1}}(\epsilon^{\lambda_1}t_1,\epsilon^{\lambda_2}t_2,m,\epsilon))(z),\nonumber\\
&=\frac{1}{(2\pi)^{1/2}}\int_{-\infty}^{\infty}\int_{L_{\xi_{h_1}}}\omega(u,m,\epsilon)\exp\left(-\left(\frac{u}{\epsilon^{\lambda_1}t_1}\right)^{k_1}-\left(\frac{u}{\epsilon^{\lambda_2}t_2}\right)^{k_2}\right)\exp(izm)\frac{du}{u}dm,\label{e472c}
\end{align}
defines a holomorphic and bounded function of $\mathcal{T}_1\times(\mathcal{T}_2\cap D(0,\rho_2))\times H_{\beta'}\times \mathcal{E}_{h_1}^0$, solving (\ref{epral}). This solution is called an \textbf{outer solution} of (\ref{epral}).

For every two consecutive outer solutions associated to (\ref{epral}), which are jointly defined in $\mathcal{T}_1\times(\mathcal{T}_2\cap D(0,\rho_2))\times H_{\beta'}\times(\mathcal{E}_{h_1}^0\cap \mathcal{E}^0_{h_1+1})$, satisfy that 
\begin{multline}\label{e495b}
\sup_{t_1\in\mathcal{T}_1,t_2\in (\mathcal{T}_2\cap D(0,\rho_2)),z\in H_{\beta'}}|u_{d_{h_1+1}}(t_1,t_2,z,\epsilon)-u_{d_{h_1}}(t_1,t_2,z,\epsilon)|\\
\le \hat{C}\exp\left(\frac{\hat{D}}{|\epsilon|^{\lambda_2k_2}}\right),\quad \epsilon\in\mathcal{E}^0_{h_1}\cap\mathcal{E}^0_{h_1+1},
\end{multline}
for two positive constants $\hat{C},\hat{D}>0$. 
\end{theo}
\begin{proof}
Let $0\le h_1\le\iota_1-1$. We proceed as in the first part of the proof of Theorem~\ref{lema495}, to arrive at the splitting  
$$u_{d_{h_1+1}}(t_1,t_2,z,\epsilon)-u_{d_{h_1}}(t_1,t_2,z,\epsilon)=E_1-E_2+E_3,$$
for every $t_1\in\mathcal{T}_1$, $t_2\in\mathcal{T}_2\cap D(0,\rho_2)$, $z\in H_{\beta'}$ and $\epsilon\in\mathcal{E}^0_{h_1}\cap\mathcal{E}^0_{h_1+1}$. According to Lemma~\ref{lema3}, statement 2.b), one has
\begin{multline}\label{e562}
|E_1|\le \hat{C}_1C_{2b}\exp\left(-\frac{(\rho/2)^{k_1}}{|\epsilon^{\lambda_1}t_1|^{k_1}}\delta_1\right)\exp\left(- \frac{(\rho/2)^{k_2}\delta_2}{2|\epsilon^{\lambda_2}t_2|^{k_2}}\right)\\
\le \hat{C}_1 C_{2b}\exp\left(-\frac{(\rho/2)^{k_2}\delta_2}{2\rho_2^{k_2}}\frac{1}{|\epsilon|^{\lambda_2k_2}}\right),
\end{multline}
for some $\hat{C}_1>0$, valid for every $\epsilon\in\mathcal{E}^0_{h_1}\cap\mathcal{E}^0_{h_1+1}$, $t_1\in\mathcal{T}_1$, $t_2\in(\mathcal{T}_2\cap D(0,\rho_2))$ and $z\in H_{\beta'}$. Analogous bounds hold for $|E_2|$. Regarding $|E_3|$, one can follow the same bounds as for $|E_1|$. We conclude (\ref{e495b}).
\end{proof}

\section{Parametric Gevrey asymptotic expansions of the solutions}\label{lastsec}

This section is devoted to the study of the asymptotic expansions associated to the outer and inner solutions of (\ref{epral}), with respect to the perturbation parameter. We make use of the cohomological criterion for $k-$summability of formal power series with coefficients in a Banach space (see~\cite{ba2}, p. 121, or \cite{hssi}, Lemma XI-2-6), known as Ramis-Sibuya theorem. We first recall the main definition of this summability theory.

\subsection{$k-$summable formal power series and Ramis-Sibuya Theorem}\label{sec51}

Let $(\mathbb{E},\left\|\cdot\right\|_{\mathbb{E}})$ be a complex Banach space.

\begin{defin} Let $k\ge1$ be an integer number. A formal power series 
$$\hat{f}(\epsilon)=\sum_{n=0}^{\infty}f_n\epsilon^n\in\mathbb{E}[[\epsilon]]$$
is said to be $k-$summable with respect to $\epsilon$ in the direction $d\in\R$ if there exists a bounded holomorphic function $f$ defined in a finite sector $V_{d}$ of bisecting direction $d$ and opening larger than $\pi/k$ and values in $\mathbb{E}$, such that it admits $\hat{f}(\epsilon)$ as its Gevrey asymptotic expansion of order $1/k$ on $V_d$, i.e. for every proper subsector $V_1$ of $V_d$, there exist $D,M>0$ with
$$\left\|f(\epsilon)-\sum_{n=0}^{N-1}f_n\epsilon^n\right\|_{\mathbb{E}}\le D M^N\Gamma(\frac{N}{k}+1)|\epsilon|^N,   $$
for every $N\ge1$ and $\epsilon\in V_1$.

Such function is unique and it is called the $k-$sum of the formal power series. Furthermore, we can reconstruct the function $f$ by means of a similar Borel-Laplace procedure as that stated in Section~\ref{seclap}.

\end{defin}

\begin{theo}[RS]
Let $(\mathcal{E}_h)_{0\le h\le \iota-1}$ be a good covering in $\C^{\star}$. For all $0\le h\le \iota-1$, let $G_h:\mathcal{E}_h\to\mathbb{E}$ be a holomorphic function and define the cocycle $\Theta_h(\epsilon):=G_{h+1}(\epsilon)-G_h(\epsilon)$, which is a holomorphic function defined in $Z_h=\mathcal{E}_{h+1}\cap \mathcal{E}_h$ into $\mathbb{E}$. We assume:
\begin{itemize}
\item[1)] $G_h$ is a bounded function as $\epsilon\in\mathcal{E}_h$ tends to $0\in\C$, for every $0\le h\le \iota-1$.
\item[2)] $\Theta_h(\epsilon)$ is an exponentially flat function of order $k$ in $Z_h$ for all $0\le h\le \iota-1$, meaning that there exist $C_h,A_h>0$ with
$$\left\|\Theta_h(\epsilon)\right\|_{\mathbb{E}}\le C_h\exp\left(-\frac{A_h}{|\epsilon|^k}\right),\qquad \epsilon\in Z_h,$$
for all $0\le h\le \iota-1$.
\end{itemize}
Then, for all $0 \leq h \leq \iota-1$, the functions $G_{h}(\epsilon)$ have a common formal power series
$\hat{G}(\epsilon) \in \mathbb{E}[[\epsilon]]$ as Gevrey asymptotic expansion of order $1/k$ on $\mathcal{E}_{h}$. Moreover,
if the aperture of one sector $\mathcal{E}_{h_0}$ is barely larger than $\pi/k$, then $G_{h_0}(\epsilon)$ is promoted as the
$k-$sum of $\hat{G}(\epsilon)$ on $\mathcal{E}_{h_0}$.
\end{theo}

\subsection{Parametric Gevrey asymptotic expansions of the inner and outer solutions of the main problem}

In this section, we display the main results of the present work, namely the asymptotic behavior of the inner and outer solutions of (\ref{epral}), constructed in the previous section. In the present section, we assume the conditions (\ref{e193})-(\ref{e220}) hold on the elements involved in the main equation under study (\ref{epral}), with forcing term $f$ determined by the construction and conditions in (\ref{e319})-(\ref{e233}). We depart from two good coverings $(\mathcal{E}^{\infty}_{h_2})_{0\le h_2\le \iota_2-1}$ of $\C^\star$, and $(\mathcal{E}^0_{h_1})_{0\le h_1\le \iota_1-1}$ of $\C^\star$, this second good covering with opening $\pi/(\lambda_2 k_2)$; and choose $\mathcal{T}_1,\mathcal{T}_2$ satisfying (\ref{e426}), (\ref{e426b}). We also fix $\chi_2^\infty$ as in Definition~\ref{defi481}, which allow us to construct the family of inner solutions associated to the first good covering. In addition to this, we choose $\rho_2>0$ and the corresponding outer solutions associated to the second good covering.

Let $\mathbb{E}_1$ denote the Banach space of holomorphic and bounded functions defined in $\mathcal{T}_1\times \chi_2^{\infty}\times H_{\beta'}$ endowed with the sup. norm, and $\mathbb{E}_2$ the Banach space of holomorphic and bounded functions in $\mathcal{T}_1\times(\mathcal{T}_2\cap D(0,\rho_2))\times H_{\beta'}$, with sup. norm.

\begin{theo}\label{lema495b}
The partial maps obtained from the inner solution of the main problem (\ref{epral}),
$$\epsilon \mapsto u_{d_{h_2}}(t_{1},\frac{x_2}{\epsilon^{\mu_2}}
e^{\sqrt{-1}\theta_{h_2}},z,\epsilon)$$
have a common formal series $\hat{u}^{\infty}(\epsilon) \in \mathbb{E}_{1}[[\epsilon]]$
as Gevrey asymptotic expansion of order $1/(\lambda_{1}k_{1})$ on $\mathcal{E}_{h_2}^{\infty}$, for $0 \leq h_2 \leq \iota_2-1$.

Each of the partial maps obtained from the outer solution of the main problem (\ref{epral}), $\epsilon \mapsto u_{d_{h_1}}(t_{1},t_{2},z,\epsilon)$ is the 
$\lambda_{2}k_{2}-$sum of a common formal power series $\hat{u}^{0}(\epsilon) \in \mathbb{E}_{2}[[\epsilon]]$ on
$\mathcal{E}_{h_1}^{0}$, for $0 \leq h_1 \leq \iota_{1}-1$.
\end{theo}

\begin{proof}
A consequence of the construction of the inner solutions in Section~\ref{secin} is that for all $0\le h_2\le \iota_2-1$, the function $\tilde{u}_{d_{h_2}}:\epsilon\mapsto u_{d_{h_2}}(t_1,\frac{x_2}{\epsilon^{\mu_2}}e^{\theta_{h_2}\sqrt{-1}},z,\epsilon)$ is a holomorphic map on $\mathcal{E}^{\infty}_{h_2}$, with values in the Banach space $\mathbb{E}_1$, for all $0\le h_2\le \iota_2-1$. Again, for every index, we consider the function $G_{h_2}:=\tilde{u}_{d_{h_2}}$ in (RS) Theorem. Due to Theorem~\ref{lema495}, we have that
$$\left\|G_{h_2+1}(\epsilon)-G_{h_2}(\epsilon)\right\|_{\mathbb{E}_1}\le\tilde{C}\exp\left(\frac{\tilde{D}}{|\epsilon|^{\lambda_1k_1}}\right),$$
for every $\epsilon\in\mathcal{E}^{\infty}_{h_2}\cap\mathcal{E}^{\infty}_{h_2+1}$. We apply (RS) Theorem in order to achieve the existence of a common formal power series $\hat{u}^{\infty}(\epsilon)\in\mathbb{E}_1[[\epsilon]]$ such that $\epsilon\mapsto u_{d_{h_2}}(t_1,\frac{x_2}{\epsilon^{\mu_2}}e^{\theta_{h_2}\sqrt{-1}},z,\epsilon)$ admits $\hat{u}^{\infty}(\epsilon)$ as its Gevrey asymptotic expansion of order $1/(\lambda_1 k_1)$, on $\mathcal{E}^{\infty}_{h_2}$, for every $0\le h_2\le\iota_2-1$.

For the second part of the proof, we consider the functions $\check{u}_{d_{h_1}}:\epsilon\mapsto u_{d_{h_1}}(t_1,t_2,z,\epsilon)$, for $0\le h_1\le\iota_1-1$, which turns out to be a holomorphic map on $\mathcal{E}^{0}_{h_1}$, with values in the Banach space $\mathbb{E}_2$. We put $G_{h_1}:=\check{u}_{d_{h_1}}$ in (RS) Theorem. Due to Theorem~\ref{lema495b}, we have that
$$\left\|G_{h_1+1}(\epsilon)-G_{h_1}(\epsilon)\right\|_{\mathbb{E}_1}\le\tilde{C}\exp\left(\frac{\tilde{D}}{|\epsilon|^{\lambda_2k_2}}\right),$$
for every $\epsilon\in\mathcal{E}^{0}_{h_1}\cap\mathcal{E}^{0}_{h_1+1}$. We apply (RS) Theorem in order to achieve the existence of a common formal power series $\hat{u}^{0}(\epsilon)\in\mathbb{E}_2[[\epsilon]]$ such that $\epsilon\mapsto u_{d_{h_1}}(t_1,t_2,z,\epsilon)$ admits $\hat{u}^{0}(\epsilon)$ as its Gevrey asymptotic expansion of order $1/(\lambda_2 k_2)$, on $\mathcal{E}^{0}_{h_1}$, for every $0\le h_1\le\iota_1-1$. Watson's Lemma (see~\cite{ba}) guarantees that this function is unique with this property, leading to a summability result. 
\end{proof}


\section{Proof of Lemma~\ref{lema3}}
It is straight to check statement 1) in Lemma~\ref{lema3}, due to $k'<k_2$.

We proceed to give proof to 2.a): let $\rho_2^\infty,\rho_1>0$. One has that
$$L_2(|T_1|,|T_2|)\le e^{-\frac{\rho^{k_1}}{|T_1|^{k_1}}\delta_1}\int_\rho^\infty e^{\nu r^{k'}}e^{-\frac{r^{k_2}}{|T_2|^{k_2}}\delta_2}dr\le e^{-\frac{\rho^{k_1}}{|T_1|^{k_1}}\delta_1}\int_0^\infty e^{\nu r^{k'}}e^{-\frac{r^{k_2}}{|T_2|^{k_2}}\delta_2}dr.$$
Let 
$$\mathcal{L}(x)=\int_0^\infty e^{\nu r^{k'}}e^{-\frac{r^{k_2}}{x}}dr.$$
We focus on determining upper bounds for $\mathcal{L}(x)$, as $x\to\infty$. Such bounds have already been observed in the proof of Proposition 4 of our recent contribution~\cite{lama19}. We provide a complete set of arguments for the sake of better readability. 

By dominated convergence, and the series representation of $\exp(\nu r^{k'})$ we derive that
$$\mathcal{L}(x)=\int_0^{\infty}\sum_{n\ge0}\frac{(\nu r^{k'})^n}{n!}e^{-\frac{r^{k_2}}{x}}dr=\sum_{n\ge0}\frac{\nu^n}{n!}\int_0^{\infty}(r^{k'})^ne^{-\frac{r^{k_2}}{x}}dr,$$
for all $x>0$.
Let 
$$\mathcal{L}_n(x)=\int_0^{\infty}(r^{k'})^ne^{-\frac{r^{k_2}}{x}}dr.$$
The change of variable $r^{k_2}/x=\tilde{r}$ yields
$$\mathcal{L}_n(x)=x^{\frac{k'}{k_2}n}x^{\frac{1}{k_2}}\frac{1}{k_2}\int_0^{\infty}(\tilde{r})^{\frac{k'}{k_2}n+\frac{1}{k_2}-1}e^{-\tilde{r}}d\tilde{r}=\frac{1}{k_2}x^{\frac{1}{k_2}}x^{\frac{k'}{k_2}n}\Gamma\left(\frac{k'}{k_2}n+\frac{1}{k_2} \right).$$
As a result, 
$$\mathcal{L}(x)=\frac{1}{k_2}x^{\frac{1}{k_2}}\sum_{n\ge0}\frac{(\nu x^{\frac{k'}{k_2}})^n}{n!}\Gamma\left(\frac{k'}{k_2}n+\frac{1}{k_2} \right)=\frac{1}{k_2}x^{\frac{1}{k_2}}\sum_{n\ge0}(\nu x^{\frac{k'}{k_2}})^n\frac{\Gamma\left(\frac{k'}{k_2}n+\frac{1}{k_2} \right)}{\Gamma(n+1)}.$$
We recall (see Appendix B in~\cite{ba2}) the Beta integral formula
$$B(\alpha,\beta)=\int_0^1(1-t)^{\alpha-1}t^{\beta-1}dt=\frac{\Gamma(\alpha)\Gamma(\beta)}{\Gamma(\alpha+\beta)}\le 1,\qquad \alpha,\beta\ge1.$$
Regrading the previous formula, we deduce the existence of a constant $C_1>0$ such that 
$$\Gamma\left(\frac{k'}{k_2}n+\frac{1}{k_2}\right)\Gamma\left(\left(1-\frac{k'}{k_2}\right)n+\left(1-\frac{1}{k_2}\right)\right)\le C_1\Gamma(n+1),\qquad n\ge \max\left\{\frac{k_2-1}{k'},\frac{1}{k_2-k'}\right\}.$$
This yields the existence of $\tilde{C}_1>0$ such that
$$\mathcal{L}(x)\le \tilde{C}_1x^{\frac{1}{k_2}}\sum_{n\ge0}\frac{(\nu x^{\frac{k'}{k_2}})^n}{\Gamma\left(\left(1-\frac{k'}{k_2}\right)n+\left(1-\frac{1}{k_2}\right)\right)}.$$
We recall the next bounds on the generalized Mittag-Leffler function (Wiman function), see~\cite{erdelyi}:
$$E_{\alpha,\beta}(z)=\sum_{n\ge0}\frac{z^n}{\Gamma(\beta+\alpha n)},\qquad \alpha,\beta>0, \alpha\in(0,2).$$
There exists $C_2>0$ such that
\begin{equation}\label{e502}
E_{\alpha,\beta}(z)\le C_2z^{\frac{1-\beta}{\alpha}}e^{z^{\frac{1}{\alpha}}},\qquad z\ge1.
\end{equation}
Regarding (\ref{e502}), we guarantee the existence of $C_3>0$ with
$$\mathcal{L}(x)\le C_3x^{\frac{1}{k_2}+\frac{1}{k_2-k'}}e^{\nu^{\frac{1}{1-\frac{k'}{k_2}}}x^{\frac{k'}{k_2-k'}}},\qquad x\ge 1.$$
As a result, we derive the formula in the statement 2.a).

We proceed to give upper bounds on $L_2(|T_1|,|T_2|)$ in the case 2.b). We assume that $|T_1|<\rho_1$ for some $\rho_1>0$, and $|T_2|<\rho_2$ for some $\rho_2>0$.

We write $L_2(|T_1|,|T_2|)=\exp(-\frac{\rho^{k_1}}{|T_1|^{k_1}}\delta_1)L_{2.1}(|T_2|^{k_2})$, where
$$L_{2.1}(|T_2|^{k_2})=\int_\rho^{\infty}e^{\nu r^{k'}}e^{-\frac{r^{k_2}}{|T_2|^{k_2}}\delta_2}dr.$$
We have
\begin{align*}
L_{2.1}(|T_2|^{k_2})&= \int_\rho^{\infty}e^{\nu r^{k'}}e^{-\frac{r^{k_2}}{2|T_2|^{k_2}}\delta_2}e^{-\frac{r^{k_2}}{2|T_2|^{k_2}}\delta_2}dr\\
&\le e^{-\frac{\rho^{k_2}}{2|T_2|^{k_2}}\delta_2}\int_\rho^{\infty}e^{\nu r^{k'}}e^{-\frac{r^{k_2}}{2|T_2|^{k_2}}\delta_2}dr\le e^{-\frac{\rho^{k_2}}{2|T_2|^{k_2}}\delta_2}\int_\rho^{\infty}e^{\nu r^{k'}}e^{-\frac{r^{k_2}}{2\rho_2^{k_2}}\delta_2}dr,
\end{align*}
being the last integral upper bounded by a positive constant, for $k_2>k'$. In conclusion, there exists some $C_{2.1}>0$ such that the statement in 2.b) holds.

\end{document}